\documentclass[11pt]{amsart}
\usepackage[totalwidth=440pt, totalheight=640pt]{geometry}
\usepackage{amsmath, amssymb, amsthm, amsfonts, 
eucal, enumerate, natbib, layout}
\usepackage[normalem]{ulem}
\usepackage{verbatim}
\usepackage{mathtools}
\usepackage{hyperref}
\hypersetup{colorlinks=true,citecolor={red},linkcolor={blue},urlcolor={violet}}

\newcommand{\details}[1]{}

\newcommand{\hookdownarrow}{\mathrel{\rotatebox[origin=c]{-90}{$\hookrightarrow$}}}
\newcommand{\longuparrow}{\mathrel{\rotatebox[origin=c]{90}{$\longrightarrow$}}}
\newcommand{\longdownarrow}{\mathrel{\rotatebox[origin=c]{-90}{$\longrightarrow$}}}
\newcommand{\longrightarrowdbl}{\rlap{$\longrightarrow$}\mathrel{\mkern-1mu}
\longrightarrow}

\newtheorem{theorem}{Theorem}[section]
\newtheorem*{theorem*}{Theorem}
\newtheorem{corollary}[theorem]{Corollary}
\newtheorem*{corollary*}{Corollary}
\newtheorem{lemma}[theorem]{Lemma}
\newtheorem*{lemma*}{Lemma}

\newtheorem*{claim*}{Claim}
\newtheorem{construction}[theorem]{Construction}
\newtheorem{proposition}[theorem]{Proposition}
\newtheorem*{proposition*}{Proposition}
\newtheorem{conjecture}[theorem]{Conjecture}
\newtheorem*{conjecture*}{Conjecture}
\newtheorem{def-proposition}[theorem]{Definition-Proposition}
\theoremstyle{definition}
\newtheorem{definition}[theorem]{Definition}
\newtheorem*{definition*}{Definition}
\newtheorem{remark}[theorem]{Remark}
\newtheorem{notation}[theorem]{Notation}

\newtheorem*{example*}{Example}
\numberwithin{equation}{section}

\newcommand{\ZZ}{\mathbb{Z}}
\newcommand{\QQ}{\mathbb{Q}}
\newcommand{\RR}{\mathbb{R}}
\newcommand{\CC}{\mathbb{C}}
\newcommand{\GG}{\mathbb{G}}
\newcommand{\PP}{\mathbb{P}}

\newcommand{\End}{\mathrm{End}}
\newcommand{\Hom}{\mathrm{Hom}}
\newcommand{\Ext}{\mathrm{Ext}}

\newcommand{\uHom}{\underline{\mathrm{Hom}}}

\newcommand{\UR}{\mathrm{UR}}
\newcommand{\W}{\mathrm{W}}

\newcommand{\HH}{\mathrm{H}}

\newcommand{\Lie}{\mathrm{Lie}\,}

\newcommand{\Gr}{\mathrm{Gr}}
\newcommand{\rk}{\mathrm{rk}}

\newcommand{\dR}{\mathrm{dR}}
\newcommand{\sing}{{\mathrm{sing}}\,}
\newcommand{\Galmot}{{\mathcal{G}}{\mathrm{al}}_{\mathrm{mot}}}

\newcommand{\tlog}{\widetilde{\log}}
\newcommand{\tE}{\widetilde{E}}
\newcommand{\Tr}{\mathrm{Tr}}

\newcommand{\oK}{\overline K}
\newcommand{\oQQ}{\overline{\QQ}}
\newcommand{\gal}{{\mathrm{Gal}}(\oK/K)}

\newcommand{\cP}{\mathcal{P}}
\newcommand{\cE}{\mathcal{E}}

\newcommand{\cL}{\mathcal{L}}

\begin{document}

\title[Semi-abelian analogues of Schanuel Conjecture]
{Semi-abelian analogues of Schanuel Conjecture and applications}

\author{Cristiana Bertolin, Patrice Philippon, Biswajyoti Saha and Ekata Saha}

\address{Dipartimento di Matematica, Universit\`a di Padova, Via Trieste 63, Padova, Italy}
\email{cristiana.bertolin@unipd.it}

\address{\'Equipe de Th\'eorie des Nombres, Institut de Math\'ematiques de Jussieu-Paris Rive Gauche,
UMR CNRS 7586, Paris, France}
\email{patrice.philippon@upmc.fr}

\address{Department of Mathematics, Indian Institute of Technology Delhi, New Delhi 110016, India}
\email{biswajyoti@maths.iitd.ac.in}

\address{Department of Mathematics, Indian Institute of Technology Delhi, New Delhi 110016, India}
\email{ekata@maths.iitd.ac.in}






\begin{abstract}
In this article we study semi-abelian analogues of Schanuel conjecture. As showed by the first author, Schanuel conjecture is equivalent to the Generalized Period conjecture applied to 1-motives without abelian part. Extending her methods, the second, the third and the fourth authors have introduced the abelian analogue of Schanuel conjecture as the Generalized Period conjecture applied to 1-motives without toric part. As a first result of this paper, we define the semi-abelian analogue of Schanuel conjecture as the Generalized Period conjecture applied to 1-motives. 

C. Cheng et al. proved that Schanuel conjecture implies the algebraic independence of the values of the iterated exponential and the values of the iterated logarithm, answering a question of M. Waldschmidt. The second, the third and the fourth authors have investigated a similar question in the setup of abelian varieties: the Weak Abelian Schanuel conjecture implies the algebraic independence of the values of the iterated abelian exponential and the values of an iterated generalized abelian logarithm. The main result of this paper is that a Relative Semi-abelian conjecture implies the algebraic independence of the values of the iterated semi-abelian exponential and the values of an iterated generalized semi-abelian logarithm.

\end{abstract}


\maketitle


\tableofcontents

\section{Introduction}

In \cite{C} C. Cheng et al. proved that Schanuel conjecture implies the algebraic independence of the values of the iterated exponential and the values of the iterated logarithm, answering a question of M. Waldschmidt. In \cite{PSS} the second, the third and the fourth authors have investigated Waldschmidt's question in the context of abelian varieties: more precisely, they first introduce after G.Vall\'ee \cite{V18} an abelian analogue of Schanuel conjecture (\emph{see} Conjecture \ref{A}) and a weak version of it, as a consequence of the Generalized Period conjecture  applied to 1-motives without toric part. Secondly they show that the Weak Abelian conjecture (\emph{see} Conjecture \ref{WA}) implies the algebraic independence of iterated values of an abelian exponential on one hand and iterated values of a corresponding generalized abelian logarithm on the other hand. 
 
 The aim of this paper is to generalize further Waldschmidt's question in the setup of semi-abelian varieties, which are extensions of abelian varieties by tori. Consider such an extension $G$ of an abelian variety $A$ by the torus $\GG_m^s$ defined over a sub-field $K$ of $\CC$. As in the case of abelian varieties, we have a 
\emph{semi-abelian exponential} map given explicitly in \eqref{eq:semiabexp}:
 \[ \exp_{G}:  \Lie G_\CC  \longrightarrow   G (\CC).\]
 An inverse map for this semi-abelian exponential, called a \emph{semi-abelian logarithm} $\log_G$ for $G$, is well-defined only locally and it might be expressed by integrating differentials of the first and the third kind on the abelian variety $A$ (see  Definition \ref{eq:log_G}):
 \[	\log_G: G (\CC)  \longrightarrow  \Lie G_\CC \cong \Lie A_\CC \times  (\Lie \GG_{m,\CC})^s. \]
Involving also the integration of differential forms of the second kind on $A$, we get a \emph{generalized semi-abelian logarithm} $\tlog_G$ for $G$ (\emph{see} Definition \ref{deftlog_G}):
\[	\tlog_G: G (\CC)  \longrightarrow  \Lie A_\CC  \times \overline{\Lie A_\CC} \times  (\Lie \GG_{m,\CC})^s. \]
  With this notation we state the Semi-abelian analogue of Schanuel conjecture \ref{SA} as the Generalized Period conjecture \ref{eq:GPC} applied to 1-motives
. Then we introduce two relative versions of conjecture \ref{SA}, called the Relative Semi-abelian conjecture \ref{WSA-V2} and the Explicit Relative Semi-abelian conjecture \ref{WSA} respectively, \emph{see} also Conjecture \ref{WSA-V1} below. We show in Proposition \ref{lem:equivWSA-V2-GPC} that they are both consequences of the Generalized Period conjecture \ref{eq:GPC} that we recall now. 
  
Let ${\overline \QQ}$ be the algebraic closure of $\QQ$ in the field of complex numbers $\CC$ and let $K$ be a sub-field of $\CC$. Consider a smooth and quasi-projective algebraic variety $X$ defined over $K$. The \textit{periods of X} are the coefficients of the matrix which represents (with respect to $K$-bases) the canonical isomorphism between the algebraic De Rham cohomology $\HH^*_{\dR}(X)$ and the singular cohomology $\HH^*_{\sing}(X(\CC),K)= \HH^*_{\sing}(X(\CC),\QQ) \otimes_\QQ K$ of $X$,
given by the integration of differentials forms
\begin{align}\label{eq:betaX}
\beta_X: \HH^*_{\dR}(X) \otimes_K \CC & \longrightarrow  \HH^*_{\sing}(X(\CC),K) \otimes_K \CC\\
\nonumber	\omega & \longmapsto \Big[ \gamma \mapsto \int_\gamma \omega \Big]
.\end{align}

In Note 10 of \cite{G66}, assuming $X$ to be a curve of any genus or an abelian variety and $K=\overline{\QQ}$, A. Grothendieck has conjectured that \textit{any polynomial relation with rational coefficients between the periods of $X$ should have a geometrical origin}. More precisely, any algebraic cycle on $X$ and on the products of $X$ with itself, will give rise to a polynomial relation with rational coefficients among the periods of $X$. 
Grothendieck has never published a precise statement of this conjecture. After early works in the context of 1-motives and their Mumford-Tate groups (by P. Deligne, J.L. Brylinski, ...), Y. Andr\'e wrote down a precise statement for motives $M$ and their \textit{motivic Galois groups} $\Galmot (M)$  recalled in \S \ref{motivicGaloisgroups} (\emph{see} \cite{A04} and  Andr\'e's letter \cite[Appendix 2]{B19} for a nice overview about this beautiful conjecture):

\begin{conjecture}[Grothendieck Period conjecture]\label{eq:CP}
	Let $M$ be a (pure or mixed) motive defined over
	$ {\overline \QQ},$ then
	\[
	\mathrm{tran.deg}_{\QQ}\, {\overline \QQ}
	\big(\mathrm{periods}(M)\big)= \dim \Galmot (M).
	\]
\end{conjecture}
\noindent Here ${\overline \QQ} (\mathrm{periods}(M))$ is the field generated over ${\overline \QQ}$ by the periods of $M$. More generally, for $S$ a set, a vector or a matrix we will denote $K(S)$ the field generated over a field $K$ by the elements or entries of $S$.

In the case of a motive associated to an algebraic variety $X$, the dimension of the motivic Galois group  of this motive is strictly related to the existence of algebraic cycles on $X$ and on the products of $X$ with itself.

 The dimension of the motivic Galois group bounds the transcendence degree from above (\emph{see} Andr\'e's letter \cite[Appendix 2]{B19}), hence Conjecture \ref{eq:CP} is really about a lower bound. 
Andr\'e has proposed the following conjecture which, thanks to results of P. Deligne \cite[Proposition 1.6]{D82} and Andr\'e \cite[\S.0.4]{A17}, is a generalization of Grothendieck Period conjecture \ref{eq:CP}:

\begin{conjecture}[Generalized Period conjecture]\label{eq:GPC}
	Let $M$ be a (pure or mixed) motive defined over a sub-field $K$ of $\CC$, then
	\[
	\mathrm{tran.deg}_{\QQ}\, K  \big(\mathrm{periods}(M)\big) \geq \dim \Galmot (M).
	\]
\end{conjecture}

In her PhD thesis, the first author showed that Schanuel conjecture \ref{SC} {\it{is equivalent}} to the Generalized Period conjecture \ref{eq:GPC} applied to 1-motives without abelian part,  $M=[u:{\ZZ}^n \to \GG_m]$, see \cite[Cor 1.3 and \S 3]{B02}.

Here, considering algebraic independence over the field generated by the periods of $G$, we show that the following conjecture is a consequence of the Generalized Period conjecture.

\begin{conjecture}[Relative Semi-abelian conjecture]\label{WSA-V1}
Let $A$ be an abelian variety defined over $\oQQ$. Let $G$ be an extension of $A$ by the torus $\GG_m^s$, which corresponds to a point $Q=(Q_1, \dots, Q_s) \in A^*(\oQQ)^s$, and denote $\Omega_G$ its matrix of periods. If $R=(R_1,\dots,R_n)$ is a point in $G(\CC)^n$ above a point $P=(P_1,\dots,P_n)\in A(\CC)^n$, then 
\[\mathrm{tran.deg}_{\QQ (\Omega_G)}\, \oQQ \big(\Omega_G, R, \tlog_{G^n}(R) \big) \geq 2\dim B_Q + \dim Z(1)\]
where $B_Q$ is $B\cap(A^n\times\{Q\})$ with $B$ the smallest algebraic subgroup of $A^n\times {A^*}^s$ containing the point $(P,Q)$, and $Z(1)$ the algebraic subgroup of $\GG_m^{ns}$ introduced in Construction \ref{B-Z}(3).
\end{conjecture}
  
 \medskip

Then, our main result is that the Relative Semi-abelian conjecture \ref{WSA-V1} (also cited in the text as conjecture \ref{WSA-V2}) implies the algebraic independence of the values of iterated semi-abelian exponential and the values of iterated generalized semi-abelian logarithms. To state explicitly our main Theorem we use the notion of algebraic independence of fields, which  for algebraically closed fields coincides with that of linear disjointness, \emph{see} for example \cite{L}, \cite{C} and \cite[Lemma 1]{PSS}. 
\begin{definition} Let $F$ be a field and $F_1,F_2$ be two  extensions of $F$ contained in a larger field $L$. The field $F_1$ is said to be \textit{algebraically independent} (or \textit{free}) from $F_2$ over $F$ if any finite set of elements of $F_1$ that is algebraically independent over $F$, remains   algebraically independent over $F_2$ (as a subset of $L$).
\end{definition}
This definition, which seems non-symmetric, is actually symmetric in $F_1$ and $F_2$. 
If $F_1$ and $F_2$ are algebraically closed and algebraically independent over $F$, then $F_1 \cap F_2 = F$.

\begin{notation} 
 As for the field $\QQ$, we will denote with $\overline{K}$ the algebraic closure in $\CC$ of any sub-field $K$ of $\CC$.
\end{notation}

Now, as in Conjecture \ref{WSA-V1}, denote $\Omega_G$ the set of periods of $G$ and consider the following two towers of algebraically closed fields:
\begin{align}
\label{eq:En}
&{\cE}_0=\overline{\QQ}, &{\cE}_n &= \overline{{\cE}_{n-1}\Big(\exp_G(r): r \in \Lie {G_\CC}(\cE_{n-1})\Big)} &&\mathrm{for} \; n \geq 1,\\
\label{eq:Ln}
&{\cL}_0 = \overline{\QQ}, &{\cL}_n &= \overline{{\cL}_{n-1}\Big( \; \widetilde{\log}_G(R): R\in G({\cL}_{n-1})\Big)} &&\mathrm{for} \; n \geq 1.
\end{align}
Let us define
\[ {\cE} = \bigcup_{n \geq 0} {\cE}_n \quad \mathrm{and} \qquad
{\cL} = \bigcup_{n \geq 0} {\cL}_n . \] 
as two fields extensions of $\oQQ$ contained in the larger field $\CC.$ Our main Theorem says:

\begin{theorem}\label{MainThm}
	If the Relative Semi-abelian conjecture \ref{WSA-V1} (see also \ref{WSA-V2}) is true, then $ {\cE}$ and $ {\cL}$ are algebraically independent over $\oQQ.$ 
\end{theorem}

Since ${\cE}$ and ${\cL}$ are algebraically closed over $\oQQ $, by \cite[Lemma 1]{PSS} we deduce from the above Theorem \ref{MainThm} that the fields $\cE$ and $\cL$ are linearly disjoint over $\oQQ$.

\medskip

After introducing generalized semi-abelian logarithms in Section \ref{semiabelian} and relevant motives with their motivic Galois groups in Section \ref{motivesandtorsors}, we discuss semi-abelian analogues of Schanuel conjecture in Section \ref{semiabelianSchanuel}. The proof of Theorem \ref{MainThm} is given in Section \ref{proofMainThm} and in Section \ref{application} we study in details the case of semi-elliptic surfaces. We end Section \ref{application} correcting the table of dimensions of the motivic Galois group of the 1-motive $M=[u:\ZZ \rightarrow G]$, where $G$ is an extension of an elliptic curve by $\GG_m$, given in \cite[\S 4]{B19}, see Subsection \ref{Rk-tableau}.

\section*{Acknowledgements}

We would like to thank the referee for her/his enlightening report and the editor for his kind consideration. The authors also thank Daniel Bertrand, Yves Andr\'e and Michel Waldschmidt for very helpful discussions and comments during the preparation of this text. Finally, we thank Pierre Deligne for writing us a clarifying letter.

The first author has worked on this paper during a 2 months stay at the Centro di Ricerca Matematica Ennio De Giorgi: she thanks this institution for the wonderful work conditions. The three last authors thanks the LIA ``Indo-French Program for Mathematics'' of CNRS for supporting their exchanges. The research of the last author is also supported by an IIT Delhi SEED grant.
 

\section{A generalized semi-abelian logarithm}\label{semiabelian}

Let $A$ be a $g$-dimensional abelian variety defined over a sub-field $K$ of $\CC$. Denote by $A_\CC$ the abelian variety obtained from $A$ extending the scalars from $K$ to the field of complex numbers. We will use similar notations $G_\CC$, $\GG_{m,\CC}$ for the other groups appearing.

\subsection{Differential forms of the first kind} Let $\omega_1, \dots , \omega_g $ be invariant differentials of the first kind on $A$ spanning the $g$-dimensional $K$-vector space $\HH^0(A, \Omega^1_A)$ of holomorphic differentials. Consider the abelian exponential $\exp_A:\Lie A_\CC\to A(\CC)$, whose kernel is the lattice $\HH_1(A(\CC),\ZZ)$ of rank $2g$ in the $g$-dimensional complex vector space $\Lie A_\CC$. Denote $z_i$, $i=1,\dots,g$, holomorphic functions on $\Lie A_\CC$, null at $0$, such that
\begin{equation}\label{eq:1stKind}
dz_i = \exp_{A}^{*}(\omega_i),\quad i=1, \dots,g.
\end{equation} 

If $\{\theta_0( z) , \theta_1( z),\dots, \theta_d ( z) \}$ is a basis of the global sections of $L^{\otimes 3}$, with $L$ any positive definite line bundle on $A_\CC$ (that is a basis of theta functions for the factor of automorphy of $L^{\otimes 3}$), we get a projective embedding:
\begin{align}
\nonumber	\exp_{A}:\Lie A_\CC & \longrightarrow  A(\CC) \subseteq \PP^d(\CC)\\
\nonumber	 z & \longmapsto \exp_{A}( z) =[ \theta_0( z): \theta_1( z): \dots: \theta_d( z)].
\end{align}
  
\begin{definition} An \textit{abelian logarithm}, denoted by $\log_A$, is an inverse map of the abelian exponential which is given by integration, more precisely 
\begin{equation}\begin{matrix}\label{eq:log_A}
\log_A: &A (\CC) & \longrightarrow  &\Lie A_\CC\hfill \\
 &P & \longmapsto &\log_A(P) =\Big(\int_O^P\omega_1, \dots ,\int_O^P  \omega_g \Big)
\end{matrix}\end{equation}
 where $O$ is the neutral element for the group law of the abelian variety $A$.
\end{definition}
\noindent There are several abelian logarithms since the map (\ref{eq:log_A}) depends on the path from $O$ to $P$ that we choose.

\medskip

Fix a positive ample (Weil) divisor $D$ on $A$ not containing the origin. According to \cite{We} there exists a unique holomorphic theta function $\theta ( z)$ on $\Lie A_\CC$ whose divisor is $\exp_A^*(D)$, with value $1$ at $0$. By \cite[\S 8 Thm 1]{M74} the map $\phi_D : A \to \mathrm{Pic}^0(A) = A^*, Q \mapsto \mathrm{div} \big( \frac{\theta ( z - q)}{\theta ( z)}\big)$ is surjective, where $q$ is any logarithm of $Q$. For any $\lambda$ in the lattice $\HH_1(A(\CC),\ZZ)$ the theta function satisfies a functional equation
\begin{equation}\label{automorphyfactor}
\frac{\theta( z+\lambda)}{\theta( z)}=\alpha(\lambda)e^{\pi H( z,\lambda)+\frac{\pi}{2}H(\lambda,\lambda)}.
\end{equation}
Here $(H,\alpha)$ are the data associated to the linear class of the divisor $D$,  where $H$ is an Hermitian form on $\Lie A_{\CC}$ satisfying ${\rm Im}(H)(\lambda,\lambda')\subseteq{\ZZ}$ and $\alpha:\HH_1(A(\CC),\ZZ)\to\{\xi\in{\CC};|\xi|=1\}$ satisfies $\alpha(\lambda+\lambda')=\alpha(\lambda)\alpha(\lambda')e^{{\rm i}\pi{\rm Im}(H)(\lambda,\lambda')}$ for $\lambda,\lambda'\in \HH_1(A(\CC),\ZZ)$.
Starting from the above theta function we construct differentials of the second and the third kind on $A$.

\subsection{Differential forms of the second kind}\label{secseckind} For $i=1, \dots,g,$ consider the function
\begin{equation}\label{eq:2ndKind-bis}
 h_i ( z) = \frac{d}{dz_i}\big(\log(\theta( z))\big) = \frac{1}{\theta( z)} \frac{d}{dz_i}\big( \theta( z)\big).
\end{equation}
Its differential $dh_i(z)$ is the pull-back by $\exp_A$ of a meromorphic form on $A$, because the function $h_i(z+\lambda)-h_i(z)$ does not depend on $z$ by \eqref{automorphyfactor}, for any $\lambda\in\ker(\exp_A)$. It also shows that for any $p\in\Lie A_\CC$ the function $h_i(z+p)-h_i(z)$ is periodic with respect to the lattice $\HH_1(A(\CC),\ZZ)$ and it is thus the pull-back of a rational function on $A$. In particular $dh_i(z)$ is a differential form of the second kind and
by formula
\begin{equation}\label{eq:exp-eta}
\exp_A^*(\eta_i)=dh_i(z)
\end{equation}
we define a differential form of the second kind $\eta_i$ on $A$, whose cohomology class in the first De Rham cohomology group $\HH^1_\dR(A)$ of the abelian variety $A$, is invariant by translations.
Observe that $h_i(z+\lambda)-h_i(z) = \int_{\gamma_\lambda}\eta_i$, where $\gamma_\lambda$ is the image by $\exp_A$ of a path from $0$ to $\lambda$ in $\mathrm{Lie}A_{\CC}$. We have that $\eta_1,\dots,\eta_g$ span the space of differential forms of the second kind modulo exact and holomorphic ones. Then, $\omega_1, \dots , \omega_g , \eta_1, \dots , \eta_g$ is a basis of  $\HH^1_\dR(A)$, which is a $2g$-dimensional $K$-vector space. Recall that if $K= \CC$ then $\HH^1_\dR(A)=\HH^0(A, \Omega^1_A) \oplus \HH^1(A, \mathcal{O}_A) $, where $\HH^1(A, \mathcal{O}_A)$ is isomorphic to the  $\CC$-vector space  of differential forms of the second kind modulo holomorphic and exact ones.

\medskip

Differential forms of the second kind on $A$ are closely related to the universal vector extension of $A$. Denote by $A^\natural$ this universal vector extension of the $g$-dimensional abelian variety $A$. The extension $A^\natural$ is not the product $A \times \overline{\Lie A_\CC}$, where $\overline{\Lie A_\CC}$ is the complex conjugate of $\Lie A_\CC$ (\emph{i.e.} the \emph{antiholomorphic} tangent space of $A$ at the origin), but $A^\natural$ is birationally isomorphic to $A \times \overline{\Lie A_\CC}$, with group law defined as
$$(P,W)+_{A^\natural}(P',W') = (P+_AP',W+W'+\mathcal H(P,P')),$$
where $\mathcal H$ is a factor system of the extension $A^\natural$. Then, $\Lie A^\natural_\CC$ is isomorphic to $\Lie A_\CC\times \overline{\Lie A_\CC}$ and the factor system $\mathcal H$ satisfies
$${\mathcal H}\big(\exp_A(z),\exp_A(z')\big) = h(z+z')-h(z)-h(z'),$$
where $h:\Lie A_\CC \to \overline{\Lie A_\CC}$ is given in coordinate system as
\begin{equation}\label{eq:2ndkind-h(z)}
z=\sum_{i=1}^gz_i\frac{\partial}{\partial Z_i} \longmapsto h(z)=\sum_{i=1}^gh_i(z)\frac{\partial}{\partial\overline{Z}_i},
\end{equation}
with the functions $h_i(z)$ introduced in \eqref{eq:2ndKind-bis}. The exponential map of $A^\natural$ is
$$\begin{matrix}
\exp_{A^\natural}: &\Lie A^\natural_\CC &\longrightarrow &A^\natural(\CC)\hfill\\
&(z,w) &\longmapsto &\big(\exp_A(z),w+h(z)\big)\\
\end{matrix}$$
and its kernel is the lattice
\begin{equation}\label{eq:H_1Anatural}
\HH_1(A^\natural(\CC),\ZZ) = \sum_{j=1}^{2g} \ZZ(\omega_{1j},\dots,\omega_{gj},-\eta_{1j},\dots,-\eta_{gj}) \subset \Lie A^\natural_\CC,
\end{equation}
where $\omega_{ij} = \int_{\gamma_j}\omega_i$ and $\eta_{ij} = \int_{\gamma_j}\eta_i$ with $\gamma_1,\dots,\gamma_{2g}$ a basis of $\HH_1(A(\CC),\ZZ)$. This exponential map induces an isomorphism between $\Lie A^\natural_\CC / \HH_1(A^\natural(\CC),\ZZ)$ and $A^\natural(\CC)$. 

Let $\eta^{A^\natural}_i$, $i=1,\dots,g$, be invariant differential forms on $A^\natural$, which induce the usual invariant differential forms on $\overline{\Lie A_\CC}$ and which are defined modulo the addition of the pull-back of a differential form of the first kind on $A$. We complete $z_1,\dots,z_g$ into a basis of $\Lie A^\natural_\CC$ with the complex coordinate functions $w_1,\dots,w_g$ defined by
\begin{equation}\label{eq:exp-xi}
dw_i = \exp_{A^\natural}^*(\eta_i^{A^\natural}).
\end{equation}
The invariant differential form $\eta_i^{A^\natural}$ is  the difference between the pull-back on $A^\natural$ of an invariant differential form on $\overline {\Lie A_\CC}$ and the pull-back  on $A^\natural$ of $\eta_i$. Identifying topologically $A(\CC)$ to $A(\CC)\times\{0\}\subset A^\natural(\CC)$, and choosing a path inside $A(\CC)$ from the neutral element $O$ of $A$ to a point $P\in A(\CC)$, we verify the equality $\int_O^P \eta^{A^\natural}_i = \int_O^P \eta_i$.

\begin{definition}\label{defgenablog}
A \textit{generalized abelian logarithm}, denoted by $\tlog_A$, is a map defined as
\begin{equation}\begin{matrix}\label{eq:tlog_A}
\tlog_A: &A (\CC) & \longrightarrow &\Lie A_\CC\times\overline{\Lie A_\CC} \\
&P & \longmapsto &\tlog_A(P) &\kern-25pt=\Big( \log_A(P), \int_O^P \eta_1, \dots ,\int_O^P \eta_g \Big)\hfill \\
&&&&\kern-25pt=\Big(\int_O^P\omega_1, \dots ,\int_O^P  \omega_g , \int_O^P \eta_1, \dots ,\int_O^P \eta_g \Big)
\end{matrix}\end{equation}
 where $O$ is the neutral element for the group law of the abelian variety $A$.
\end{definition}
\noindent As in the case of abelian logarithms, there are several generalized abelian logarithms depending on the choice of the path of integration. 

\medskip

A logarithm for the universal vector extension $A^\natural$ of $A$ is induced by a generalized abelian logarithm $\tlog_A$ of $A$. Indeed, for $P\in A(\CC)$ one views $\tlog_A(P)$ as an element of $\Lie A_\CC\times\overline{\Lie A_\CC}$, identifying $\big(\int_O^P\eta_1,\dots,\int_O^P\eta_g\big)$ with $h(p)\in\overline{\Lie A_{\CC}}$ by \eqref{eq:2ndkind-h(z)}, then
$$\begin{matrix}
\log_{A^\natural}: &A^\natural (\CC) &\longrightarrow  & \Lie A^\natural_\CC \simeq \Lie A_\CC\times\overline{\Lie A_\CC} &\longrightarrowdbl &\Lie A^\natural_\CC \big/ \HH_1(A^\natural(\CC),\ZZ)\\
&(P,W) &\longmapsto &\tlog_A(P)+(0,W)\\
\end{matrix}$$
gives a locally well-defined inverse map to the complex analytic isomorphism induced by $\exp_{A^\natural}$: 
\[\Lie A^\natural_\CC \big/ \HH_1(A^\natural(\CC),\ZZ) \to A^\natural(\CC). \]

\subsection{Differential forms of the third kind} As explained in \cite[\S 9]{Br} differentials of the third kind on $A$ are closely related to extensions of $A$ by $\GG_m.$ Consider an extension $G$ of our abelian variety $A$ by the torus $\GG_m,$ which is defined over the subfield $K$ of $\CC$. Via the isomorphism $\underline{\Ext}^1(A,\GG_m) \cong \mathrm{Pic}^0(A) = A^*$, it is equivalent to having the extension $G$ or to have a divisor in $\mathrm{Pic}^0(A) $ or to have a point $Q$ in $ A^*(K)$. We now complete the basis $\omega_1,\dots,\omega_g,\eta_1,\dots,\eta_g$ of $\HH^1_\dR(A)$ into a $K$-basis of the first De Rham cohomology group $\HH^1_\dR(G)$.

Let $q$ be a logarithm in $\mathrm{Lie}A^*_{\mathbb C}$ of the point $Q$ and consider the function
\begin{equation}\label{eq:def-Fq}
F_q(z)= \frac{\theta(z+q)}{\theta(z)\theta(q)} e^{-\sum_{i=1}^g h_i(q)z_i},
\end{equation}
which generalizes Serre's function \cite[(1.6)]{B19}, see also Table \ref{tablecorr} in Section \ref{application}. Since $\frac{F_q(z+\lambda)}{F_q(z)}$ is constant in $z$ for any $\lambda\in \HH_1(A(\CC),\ZZ)$, the differential form $d\log(F_q( z))$ is the pull-back via the abelian exponential of a differential form of the third kind $\xi_Q$ on $A$ whose residual divisor is $\phi_D (-Q)$. One checks that the function $\frac{F_q( z+ z')}{F_q( z)F_q( z')}$ induces a rational map from $A^2$ to $\GG_m$, which is a factor system for the extension $G$ of $A$ by $\GG_m$ corresponding to the point $Q$. 
If $Q$ isn't a torsion point of $A^*$, the extension $G$ is not the product $A \times \GG_m$, but $G$ is birationally isomorphic to $A \times \GG_m$ with the group law twisted by the above factor system. The exponential map of $G$ is
\begin{equation}\label{eq:semiabexpi}
\begin{matrix}
\exp_{G}:  &\Lie G_\CC\cong \Lie A_\CC \times \Lie \GG_{m,\CC}\hfill  &\longrightarrow   & A(\CC)\times\CC^\times \cong G(\CC) \hfill\\
&(  z, t) &\longmapsto   &\big(\exp_A(z), e^{t}F_{q}(z)\big).\hfill\end{matrix}\end{equation}
Let $\xi^G_Q$ be an invariant differential form on $G$, which induces the usual invariant differential form on $\GG_m$ and which is defined modulo the addition of the pull-back of a differential form of the first kind on $A$. We complete $z_1,\dots,z_g$ into a basis of $\Lie G_\CC$ with the complex coordinate function $t$ defined by
\begin{equation}\label{eq:exp-xii}
dt = \exp_G^*(\xi^G_Q).
\end{equation}
Note that the pull-back $\tilde\xi_Q$ of the invariant differential form $\frac{dZ}{Z}$ on $\GG_m$, via the rational map $G \cong A \times \GG_m \to \GG_m$, is not invariant on $G$ and it satisfies
\begin{equation*}
\exp_{G}^{*}(\tilde\xi_Q) = \frac{d\big(e^{t}F_{q}(z)\big)}{e^{t}F_{q}(z)} = dt + d\log\big(F_q( z)\big) = dt + \frac{1}{F_q( z)}\sum_{i=1}^g \frac{d}{dz_i}\big(F_q( z)\big) d z_i.
\end{equation*}
We thus have $\xi^G_Q=\tilde\xi_Q-\pi^*\xi_Q$, where $\pi$ is the projection from the extension $G$ to $A$. For $R=(P,e^\ell)\in G(\CC)$ and $\exp_A(p)=P=\pi(R)$ we compute
\begin{equation}\label{eq:3rdKind}
\int_O^R\xi^G_Q = \int_1^{e^\ell}\frac{dZ}{Z} - \int_O^P\xi_Q \equiv \ell-\log\big(F_q(p)\big)
\end{equation}
modulo the kernel of $\exp_G$, since the integrals depend on the choice of a path of integration. Therefore,
\begin{equation*}\begin{matrix}
\log_G: &G (\CC) &\longrightarrow  &\Lie G_\CC \cong \Lie A_\CC \times \Lie \GG_{m,\CC}\hfill \\
&R &\longmapsto &\log_G(R)\hfill &\kern-100pt=\Big(\log_A(P), \int_O^R \xi^G_{Q}  \Big)
\end{matrix}\end{equation*}
is an inverse map of the exponential of $G$.

Consider now an extension $G$ of our abelian variety $A$ by $\GG_m^s,$ which is defined over $K$. Via the isomorphism ${A^*}^s = \mathrm{Pic}^0(A)^s \cong \underline{\Ext}^1(A,\GG_m^s)$, having the extension $G$ is equivalent to having $s$ points $Q_1, \ldots, Q_{s}$ in $ A^*(K)$. A basis of the first De Rham cohomology group  $\HH^1_\dR(G)$ of the extension $G$ is given by $\{\pi^*\omega_1, \dots , \pi^*\omega_g, \pi^*\eta_1, \dots , \pi^*\eta_g, \xi^G_{Q_1}, \dots ,\xi^G_{Q_s}\}$, where $\pi:G\to A$ is the projection underlying the extension $G$.
 
The extension $G$ is birationally isomorphic to $ A \times \GG_m^s $, thus its dimension is $g+s$ and $\Lie G_\CC $ is isomorphic to the $(g+s)$-dimensional complex vector space $\Lie A_\CC \times  (\Lie \GG_{m,\CC})^s.$ In \cite[\S 9]{Br} Bertrand has computed explicitly the semi-abelian exponential map and an embedding for $G$:
\begin{equation}\label{eq:semiabexp}
\begin{matrix}
\exp_{G}:  &\Lie G_\CC  &\longrightarrow   &G(\CC) \subseteq (\PP^{d}\times\PP^s)(\CC)\hfill\\
&(  z, t_1, \dots, t_s) &\longmapsto   &\Big(\exp_A(  z) \; , \; [e^{t_1} F_{q_1}( z): \dots : e^{t_s} F_{q_s}( z) :1 \big]
\Big)\hfill
\end{matrix}\end{equation}
whose kernel is the lattice $\HH_1(G(\CC),\ZZ)$ of rank $2g+s$ of the complex vector space $\Lie G_\CC .$

\begin{definition}\label{eq:log_G} A \textit{semi-abelian logarithm}, denoted by $\log_G$, is an inverse map of the semi-abelian exponential which is given by integration, more precisely 
\begin{equation*}\begin{matrix}
\log_G: &G (\CC) &\longrightarrow  &\Lie G_\CC \kern7pt\cong \Lie A_\CC \times  \Lie \GG_{m,\CC}^s\hfill \\
&R &\longmapsto &\log_G(R)\hfill &\kern-106pt=\Big(\log_A(P), \int_O^R \xi^G_{Q_1}, \dots ,\int_O^R \xi^G_{Q_s}  \Big)\hfill \\
&&&&\kern-106pt=\Big(\int_O^P\omega_1, \dots ,\int_O^P  \omega_g , \int_O^R \xi^G_{Q_1}, \dots ,\int_O^R \xi^G_{Q_s}  \Big)
\end{matrix}\end{equation*}
where $P=\pi(R)$ and $O$ is the neutral element for the group law of the semi-abelian variety $G$ or the abelian variety $A$, accordingly.
\end{definition}
\noindent Again, there are several semi-abelian logarithms since the map $\log_G$ in Definition \ref{eq:log_G} depends on the path from $O$ to $R$ we choose. The inverse map of the complex analytic isomorphism $\Lie G_\CC / \HH_1(G(\CC),\ZZ) \to G(\CC)$ induced by the semi-abelian exponential, is given by $G(\CC) \to \Lie G_\CC / \HH_1(G(\CC),\ZZ), R \mapsto \log_G(R)$ mod $\HH_1(G(\CC),\ZZ)$, which is well defined.

\begin{definition}\label{deftlog_G} A \textit{generalized semi-abelian logarithm}, denoted by $\tlog_G$, is a map defined by
\begin{equation*}\begin{matrix}	
\tlog_G: &G (\CC) & \longrightarrow  &\Lie A_\CC\times\overline{\Lie A_\CC} \times  \Lie \GG_{m,\CC}^s \\
&R & \longmapsto &\tlog_G(R)\hfill &\kern-100pt=\Big(\tlog_A(P), \int_O^R \xi^G_{Q_1}, \dots ,\int_O^R \xi^G_{Q_s} \Big)\hfill  \\
&&&&\kern-100pt =\Big(\int_O^P\omega_1, \dots ,\int_O^P  \omega_g , \int_O^P \eta_1, \dots ,\int_O^P \eta_g, \int_O^R \xi^G_{Q_1}, \dots ,\int_O^R \xi^G_{Q_s} \Big)
\end{matrix}\end{equation*}
 where $P=\pi(R)$ and $O$ is the neutral element for the group law of the semi-abelian variety $G$ or the abelian variety $A$, accordingly.
\end{definition} 
\noindent As before, there are several generalized semi-abelian logarithms depending on the choice of the path of integration.

Given $\gamma_1,\dots,\gamma_{2g}$ a basis of $\HH_1(A(\CC),\ZZ)$, we can lift these paths in $\HH_1(G(\CC),\ZZ)$, as $\tilde\gamma_1,\dots,\tilde\gamma_{2g}$, since topologically $G(\CC) \cong A(\CC) \times (\CC^\times )^s$. We then form \emph{period matrices} associated to $A$ and $G$:
\begin{equation}\label{eq:period-matrix-A-G}
\Omega_A := \Big(\begin{matrix}
\int_{\gamma_j}\omega_i &\int_{\gamma_j}\eta_i\end{matrix}\Big)_{ji}\enspace,\quad
\Omega_G := \left(\begin{matrix}\Omega_A &\big(\int_{\tilde\gamma_j}\xi^G_{Q_k}\big)_{jk} \\[2mm]
0 &2{\rm i}\pi{\rm Id}_{s\times s}\end{matrix}\right),
\end{equation}
where $i$ is running from $1$ to $g$, $j$ from $1$ to $2g$ and $k$ from $1$ to $s$. The entries of the matrix $\Omega_A$ are also written $\omega_{ij}$ and $\eta_{ij}$ in the sequel. They are called the \textit{periods} and \textit{quasi-periods} of $A$, respectively. Note that two generalized semi-abelian logarithms differ by $\ZZ$-linear combinations of entries of $\Omega_G$. 
\begin{remark}\label{rem:splitint}
	The $ns$ complex numbers $\int_O^{P_i}\xi_{Q_k}$ are the abelian integrals of the third kind,
which are computed explicitly in the elliptic case in \cite[Prop 2.3]{B19}. The exponentials of the numbers $\int_{\gamma_j}\xi_{Q_k}$ are called the \textit{quasi-quasi-periods} of $A$. According to formula \eqref{eq:3rdKind} the integral $\int_O^{R_i}\xi^G_{Q_k}$ breaks in two pieces $\int_1^{e^{\ell_i}}\frac{dZ_k}{Z_k}$ and $-\int_O^{P_i}\xi_{Q_k}$, where $R_i=(P_i,e^{\ell_i})$.
\end{remark}

\medskip

For convenience, in this paper we will use small letters for semi-abelian logarithms of points on $G(\CC)$ which are written with capital letters, that is $\exp_{G}(r)=R \in G (\CC)$ for any $r \in \Lie G_\CC$.


\section{1-motives and torsors}\label{motivesandtorsors}
Before we state semi-abelian analogues of Schanuel conjecture in the next section, we recall some basic facts about 1-motives, their motivic Galois groups and related torsors. 


 In this section $K$ is a sub-field of $\CC$.

\subsection{Motivic Galois groups of 1-motives}\label{motivicGaloisgroups}

\begin{definition}\label{def-motive}
A 1-motive $M=[u:X \rightarrow G]$ over an arbitrary scheme $S$ consists of:
\begin{itemize}
	\item  a group scheme $X$ over $S$, defined by a finitely generated free $\ZZ\,$-module, which is a locally constant group scheme for the \'etale topology,
	\item an extension $G$ of an abelian $S$-scheme $A$ by an $S$-torus $T_0$, which is defined over $S$,
	\item a morphism $u:X \to G$ of group schemes over $S$.
\end{itemize}

\end{definition}

In this paper we will work with $S =\mathrm{Spec} (K).$ With this hypothesis
we think of $X$ as the character group of a torus defined over $K$, that is a finitely generated $\gal$-module. 
We identify $X(\oK)$  with a finitely generated free $\ZZ\,$-module, denoted $\ZZ^{{\rk}X}.$
Since any morphism of $K$-schemes can be seen as a $\gal$-equivariant morphism  
of the corresponding $\oK$-schemes, having the group morphism $u:X \to G$ is equivalent to having a $\gal$-equivariant group morphism 
\[u:  X(\oK) \longrightarrow G(\oK)\]
 from $\ZZ^{{\rk} X}$ to the $\oK$-rationnal points of $G$.

There is a more symmetrical definition of 1-motives. In fact, having the 1-motive $M=[u:X \rightarrow G]$ is equivalent to having the 7-tuple $(X, Y^\vee,A ,A^*, v ,v^*,\psi)$ where:
\begin{itemize}
	\item $X$ is the character group of a $n$-dimensional torus;
	\item $Y^\vee$ is the character group of a $s$-dimensional torus $T_0$;
	\item $A^*\simeq \underline{\Ext}^1(A,\GG_m)$ is the Cartier dual of the abelian variety $A$;
	\item $v: X \rightarrow A$ and $v^*:Y^\vee \rightarrow A^*$ are two group morphisms of $K$-group varieties. Having the group morphism $v$ is equivalent to having $n$ points $P_1, \ldots, P_n\in A(K)$ , whereas having the group morphism $v^*$ is equivalent to having $s$ points $Q_1, \ldots, Q_s\in A^*(K)$ which parameterize the isomorphy class of the extension $G$ of $A$ by $T_0$. 
	\item $\psi$ is a trivialization of the pull-back $(v,v^*)^*\mathcal{P}$ via $(v,v^*)$ of the Poincar\'e biextension $\mathcal{P}$ of $(A,A^*)$ by $\GG_m$ (\textit{see} \cite[Chapter 2, \S 2.5, page 37 - 40]{BL}). Having this trivialization $\psi$ is equivalent to having $n$ points $R_1, \ldots, R_n \in G(K)$ whose images via the projection $G \to A$ are the $n$ points $P_1, \ldots, P_n$ respectively. Therefore having $\psi$ is equivalent to having the group morphism $u:X \rightarrow G$. 
\end{itemize}

We denote $Y(1)$ the torus whose character group is $Y^\vee$ and co-character group is $Y$. In particular, with notations as above $T_0=Y(1)$.

\medskip

Any 1-motive $M=[u:X \rightarrow G] $ is endowed with an increasing filtration $\W_{\bullet}$, called the \textit{weight filtration} of $M$:
\begin{equation}\label{eq:weight-filtration}
\W_{0}(M)=M,\quad  \W_{-1}(M)=[0 \to G],\quad \W_{-2}(M)=[0 \to Y(1)].
\end{equation}
If we set ${\Gr}_{n}^{\W} := \W_{n} / \W_{n-1},$ we have 
$${\Gr}_{0}^{\W}(M) = [ X \to 0],\quad
{\Gr}_{-1}^{\W}(M) = [0 \to A],\quad
{\Gr}_{-2}^{\W}(M) = [0 \to Y(1)].$$

Let $x_1,\dots,x_n$ be a base of $X\simeq\ZZ^n$ and $M=[u:X \rightarrow G]$ be a 1-motive defined over $K$ such that $u(x_k)=R_k \in G(K)$, $k=1,\dots,n$, where $G$ is an extension of $A$ by $\GG_m^s$ corresponding to $s$ points $Q_1,\dots,Q_s \in A^*(K)$. Denote by $P_1,\dots,P_n$ the images of the $n$ points $R_1,\dots,R_n$ via the projection $\pi:G\to A$. The periods of $M$ are, by definition, the coefficients of the matrix which represents the isomorphism between the De Rham and the Hodge realizations of $M$. That is $1$, $\omega_{ij}$, $\eta_{ij}$, $\int_{\tilde\gamma_j} \xi^G_{Q_k}, 2 \mathrm i \pi$, $\tlog_G(R_\ell)$ for $i=1, \dots, g$, $j=1, \dots, 2g$, $k=1,\dots,s$ and $\ell=1,\dots,n$, see \cite[Prop 2.3]{B19}, giving with notations \eqref{eq:period-matrix-A-G}:
\begin{equation}\label{eq:OmegaM}
\Omega_M := \left(\begin{matrix}&\tlog_G(R_1)\\ {\rm Id}_n &\vdots\\ &\tlog_G(R_n)\\ 0 &\Omega_G\end{matrix}\right).
\end{equation}

Denote by $\mathcal{MM}_{\leq 1}(K)$ the tannakian sub-category with rational coefficients generated by the 1-motives defined over $K$, in Nori's (tannakian) category of mixte motives, see \cite{N00}. The 1-motive $\ZZ=\ZZ(0):= [ \ZZ \to 0]$ induces the unit object $\QQ(0)$ of $\mathcal{MM}_{\leq 1}(K)$ and the 1-motive $\GG_m=\ZZ(1):= [ 0 \to  \GG_m ]$ induces the object $\QQ(1)$ of $\mathcal{MM}_{\leq 1}(K)$. If $M$ is a 1-motive, we denote by $M^\vee  \cong \uHom (M, \ZZ)$ its dual and by $ev_M : M \otimes M^\vee \to \ZZ$ and $\delta_M:  \ZZ \to M^\vee \otimes M$ the morphisms of $\mathcal{MM}_{\leq 1}(K)$ characterizing this dual. The Cartier dual of $M$ is $M^*= M^\vee \otimes \QQ(1) = (Y^\vee,X,A^*,A,v^*,v,\psi\circ s)$, where $s$ is the morphism which exchange the factors.

\begin{notation}
In this section, the letters $X$, $A$, $A^*$, $T_0$ stands for the 1-motives $[X\to 0]$, $[0\to A]$, $[0\to A^*]$ and $[0\to T_0]$, associated to the $\ZZ$-module $X$, the abelian varieties $A$, $A^*$ and the torus $T_0=Y(1)$, respectively. Then, the sums $+$ and tensor products $\otimes$ are taken in the category $\mathcal{MM}_{\leq 1}(K)$.
\end{notation}

If $M_1$, $M_2$, $M_3$ are 1-motives, we set
\begin{equation}\label{eq:BiextHom}
\Hom_{\mathcal{MM}_{\leq 1}(K)}(M_1 \otimes M_2, M_3):= \mathrm{Biext}^1 (M_1,M_2; M_3)
\end{equation}
where $\mathrm{Biext}^1(M_1,M_2;M_3)$ is the abelian group of isomorphism classes of biextensions of $(M_1,M_2)$ by $M_3$, see \cite[\S10.2]{D75}, \cite[Definition 3.1.1]{B09} and \cite{B09bis,BM09, BT14, BT18, B12, B13, BG, BG2}.
In particular, the isomorphism class of the Poincar\'e biextension $\mathcal{P}$ of $(A,A^*)$ by $\GG_m$ is the Weil pairing of $A$ \cite{D02}:
\begin{equation}\label{eq:Weil-pairing}
\mathcal{W}_A : A \otimes A^* \to \GG_m.
\end{equation}

The tannakian sub-category $\langle M\rangle^\otimes$ generated by the 1-motive $M$ is the full sub-category of $\mathcal{MM}_{\leq 1}(K)$ whose objects are sub-quotients of direct sums of $M^{\otimes \; n}  \otimes M^{\vee \; \otimes \; m}$, and whose fibre functor is the restriction of the fibre functor of $\mathcal{MM}_{\leq 1}(K)$ to $\langle M\rangle^\otimes$. 

\begin{definition}\label{motivic-Galois}
The \emph{motivic Galois group} $\Galmot (M)$ of $M$ is
the motivic affine group scheme ${\rm Sp}( \Lambda),$ where $ \Lambda$ is the commutative Hopf algebra  of $\langle M\rangle^\otimes$ which is universal for the following property: for any object $U$ of $\langle M\rangle^\otimes,$ there exists a morphism  $\lambda_U:  U^{\vee} \otimes U \to \Lambda$ functorial on  $U$, i.e. for any morphism $f: U \to V$ in $\langle M\rangle^\otimes$ and $f^t$ its transpose, the diagram 
 \[\begin{matrix} 
V^{\vee} \otimes U&{\buildrel f^t \otimes 1 \over \longrightarrow}&	U^{\vee} \otimes U \cr
{\scriptstyle 1 \otimes f}\downarrow \quad \quad & & \quad \quad \downarrow{\scriptstyle \lambda_U}\cr
V^{\vee} \otimes V & {\buildrel \lambda_V \over \longrightarrow} & \Lambda
\end{matrix}\]
is commutative, \emph{see} \cite[D\'efinition 8.13]{D90} and \cite[D\'efinition 6.1]{D89}.
\end{definition}

Two 1-motives $M_i=[u_i:X_i \rightarrow G_i]$ over $K$ (for $i=1,2$) are isogeneous if there exists a morphism of complexes $(f_X,f_G):M_1 \to M_2$ such that $f_X:X_1 \to X_2$ is injective with finite cokernel, and $f_G:G_1 \to G_2$ is surjective with finite kernel (\textit{see} \cite{BB}). Since \cite[Construction (10.1.3)]{D75} is true modulo isogenies, two isogeneous 1-motives have the same periods. Moreover, two isogeneous 1-motives build the same tannakian category and so they have the same motivic Galois group. Hence in this paper \textit{we can and will work modulo isogenies}. 
In particular, we identify the abelian variety $A$ with its Cartier dual $A^*$.

\medskip
The weight filtration $\W_{\bullet}$ of $M$, see \eqref{eq:weight-filtration}, induces a filtration on its motivic Galois group $\Galmot (M)$, see \cite[Chapter IV \S 2]{S72}:
\medskip
\begin{equation}\label{eq:weightonGalois}
\begin{matrix}
\hfill\W_{0}(\Galmot(M)) &= &\Galmot(M)\hfill\\[1mm]
\hfill\W_{-1}(\Galmot(M)) &= &\big\{ g \in \Galmot(M) \, \, \vert \, \, (g - id)W_n(M) \subseteq \W_{n-1}(M)\, , \,\forall n\, \big\}\hfill\\[1mm]
\hfill\W_{-2}(\Galmot(M)) &= &\big\{ g \in \Galmot(M) \, \, \vert \, \, (g - id)W_n(M) \subseteq \W_{n-2}(M) \, , \,\forall n\,\big\}\hfill\\[1mm]
\hfill\W_{-3}(\Galmot(M)) &= &\{id\}.\hfill\\[1mm]
\end{matrix}\end{equation}
Clearly $\W_{-1}(\Galmot(M))$  is unipotent. Denote by $\UR (M)$ the unipotent radical of $\Galmot(M)$ and set ${\Gr}_{n}^{\W}(\Galmot(M)) = W_n(\Galmot(M))/W_{n-1}(\Galmot(M))$. In \cite[\S 3]{B19}, the first author has proved: 
\begin{itemize}
	\item ${\Gr}_{0}^{\W}(\Galmot(M)) = \Galmot (X+A+Y(1))$ is reductive. In fact $\Galmot(X+A+Y(1))=\Galmot(A)$, 
	\item $\W_{-1}(\Galmot(M))=\ker \big[\Galmot (M) \twoheadrightarrow  \Galmot (X+A+Y(1)) \big]$ is the unipotent radical $ \UR (M)$ of $\Galmot(M).$ In particular,
\begin{equation}\label{eq:Dim}
\dim \Galmot (M) = \dim \Galmot (A) + \dim \UR (M),
\end{equation}
	\item $\W_{-2}(\Galmot(M))=\ker \big[\Galmot (M) \twoheadrightarrow   \Galmot (N)\big]$ setting
	$$N := (M + M^*) /\W_{-2} (M + M^*),$$
	\item ${\Gr}_{-1}^{\W}(\Galmot(M))$ 
is  the unipotent radical $\W_{-1} \big( \Galmot (N)\big)$  of $\Galmot\big(N\big)$:
\begin{equation} \label{eq:UR}
 0 \to \W_{-2}\big(\Galmot(M)\big) \to \UR(M) \to \UR\big(N\big) \to 0.  
 \end{equation}
\end{itemize}

\medskip

Consider the motive
\[\mathbb{E}=\W_{-1}\big( {\underline {\End}}(X+A+Y(1))\big).\]
which is the direct sum of the pure motives $\mathbb{E}_{-1}:={\Gr}_{-1}^{\W}(\mathbb{E})=(X^\vee \otimes A) + (A^* \otimes Y)$ and $\mathbb{E}_{-2}:={\Gr}_{-2}^{\W}(\mathbb{E})=(X^\vee \otimes Y)(1)$ of weight $-1$ and $-2$ respectively. Recall that we have set $\rk(X)=n$, $\rk(Y^\vee)=s$. As observed in \cite[\S 3.1, p.599]{B03}, the composition of endomorphisms furnishes a ring structure on $\mathbb{E}$ which is given by the morphism $PW: \mathbb{E} \otimes \mathbb{E} \to \mathbb{E}$ of $\langle M \rangle^\otimes$ whose only non trivial component is 
$$\mathbb{E}_{-1} \otimes \mathbb{E}_{-1} \longrightarrow (X^\vee \otimes A) \otimes (A^* \otimes Y) \longrightarrow {\ZZ}(1) \otimes X^\vee \otimes Y = \mathbb{E}_{-2},$$
where the first arrow is the projection from  $ \mathbb{E}_{-1} \otimes \mathbb{E}_{-1}$ to $(X^\vee \otimes A) \otimes (A^* \otimes Y)$ and the second arrow is $ns$ copies of the Weil pairing $\mathcal{W}_A: A \otimes A^* \to \ZZ(1)$ of $A$, see \eqref{eq:Weil-pairing} and \cite[(2.8.4), p.598]{B03}. Because of definition (\ref{eq:BiextHom}) the product $PW:\mathbb{E}_{-1} \otimes \mathbb{E}_{-1} \to \mathbb{E}_{-2}$ defines a biextension $\mathcal{B}$ of $(\mathbb{E}_{-1},\mathbb{E}_{-1})$ by $\mathbb{E}_{-2}$, whose pull-back $d^* \mathcal{B}$ via the diagonal morphism $d:\mathbb{E}_{-1} \to \mathbb{E}_{-1} \times \mathbb{E}_{-1}$ is a $\Sigma - (X^\vee \otimes Y)(1)$-torsor over $\mathbb{E}_{-1}$, see \cite[D\'efinition 5.4]{Breen}. In \cite[Example 2.8]{B03} the biextension $\mathcal{B}$ is constructed explicitly using $ns$ copies of the Poincar\'e biextension $\mathcal{P}$ of $(A,A^*)$ by $\GG_m$, $ns$ copies of the Poincar\'e biextension $\mathcal{P}^*$ of $(A^*,A)$ by $\GG_m$, $n^2$ copies of the trivial biextension of $(A,A)$ by $\GG_m$ and $s^2$ copies of the trivial biextension of $(A^*,A^*)$ by $\GG_m$. By \cite[Lem 3.3, p.600]{B03} the $\Sigma - (X^\vee \otimes Y)(1)$-torsor $ d^* \mathcal{B}$ induces a Lie bracket 
\[ [\cdot,\cdot]: \mathbb{E} \otimes \mathbb{E} \to \mathbb{E}\]
 on $\mathbb{E}$ (an explicit description of this Lie bracket is given in \cite[(2.8.4)]{B03}).
 
Observe that the motives $X^\vee\otimes A$ and $A^n$ coincide, the same remark applies to $A^*\otimes Y$ which coincide with ${A^*}^s$. For the ease of notation we will denote $T$ the motive $(X^\vee\otimes Y)(1)$. Thus the Lie bracket $[\,,\,]$ restricted to $(A^n\times{A^*}^s)\otimes (A^n\times{A^*}^s)$ has image in $T$.

As explained in \cite[\S 3]{B03},
\begin{itemize}
	\item to have the morphisms $v: X \to A$ and $v^*: Y^\vee \to A^*$ underlying the 1-motive $M$, defined by the points $P_1,\dots,P_n\in A(K)$ and $Q_1,\dots,Q_s\in {A^*}(K)$ respectively, is the same thing as to have the morphisms
\begin{equation}\label{defV-V*}
\begin{matrix}
V: &\ZZ &\to &A^n\hfill\\
&\hfill 1 &\mapsto &P:=(P_1,\dots,P_n)
\end{matrix}
\qquad \mbox{and}\qquad 
\begin{matrix}
V^*: &\ZZ &\to &{A^*}^s\hfill\\
&\hfill 1 &\mapsto &Q:=(Q_1,\dots,Q_s),
\end{matrix}
\end{equation}
i.e. to have a point $(P,Q)$ of $ \mathbb{E}_{-1}(K)= (A^n\times {A^*}^s)(K)$. 
	\item having the morphism $u: X \to G$, that is having $n$ points $R=(R_1,\ldots, R_{n})$ of $G$, whose images via the projection $G \to A$ are the points $P_1, \ldots, P_{n}$ respectively, is equivalent to having a point 
\begin{equation}\label{R}
\widetilde{R} \in (d^*\mathcal{B})_{(P,Q)}
\end{equation}
in the fibre of $d^*\mathcal{B}$ over the point $(P,Q)$ of $\mathbb{E}_{-1}(K)=(A^n\times {A^*}^s)(K)$. Recall that a point on a bundle is defined modulo its factor of automorphy. The bundle is split whenever this factor is trivial.
\end{itemize}

\begin{construction}\label{B-Z}
We will consider the following pure motives $B$, $Z'(1)$ and $Z(1)$, associated to the 1-motive $M$:
\begin{enumerate}
\item $B$ is the \textit{smallest} abelian sub-variety (modulo isogenies) of  $A^n\times {A^*}^s$ which contains the point $(P,Q) \in (A^n\times {A^*}^s)(K) $. The pull-back
\begin{equation} \label{eq:idB}
i^*d^* \mathcal{B}
\end{equation}
of $d^* \mathcal{B}$ via the inclusion $i: B \hookrightarrow A^n\times {A^*}^s$, is a $\Sigma-T$-torsor over $B$, see \cite[D\'efinition 5.4]{Breen}. 
\item $Z'$ is the \textit{smallest} $\gal$-sub-module of $X^\vee \otimes Y$  such that the sub-torus $Z'(1)$ of $T$ contains the image of the Lie bracket $[\cdot,\cdot]: B \otimes B \to T$. The push-down ${pr}_*i^*d^* \mathcal{B}$ of the $\Sigma-T$-torsor $i^*d^* \mathcal{B}$ via the projection
$pr:T \twoheadrightarrow T/Z'(1)$ is the trivial $\Sigma-T/Z'(1)$-torsor over $B$, i.e. 
\[{pr}_*i^*d^* \mathcal{B}= B \times(T/Z'(1)).\] Denote $\pi: {pr}_*i^*d^* \mathcal{B} \twoheadrightarrow T/Z'(1)$ the canonical projection. We still denote $\widetilde{R}$ the points of $i^*d^* \mathcal{B}$ and of ${pr}_*i^*d^* \mathcal{B}$ living over $(P,Q) \in B$.
\item $Z$ is the \textit{smallest} $\gal$-sub-module of $X^\vee\otimes Y$ containing $Z'$ and such that the sub-torus $(Z/Z')(1)$ of $T/Z'(1)$ contains $\pi (\widetilde{R})$.
\end{enumerate}
We summarize the construction in the following diagram:
   \begin{equation*}
  \begin{array}{ccccccccc}
  T &= &  T &= &T &\stackrel{pr}{\twoheadrightarrow}&T/Z'(1) &\hookleftarrow&(Z/Z')(1)\\[1mm]
  \circlearrowright && \circlearrowright &&\circlearrowright &&\;\longuparrow\rlap{$\pi$} &&\longuparrow\\[2mm]
  \mathcal{B} &  \hookleftarrow & d^*\mathcal{B}  & \hookleftarrow  & i^*d^*\mathcal{B} &  \longrightarrow &B\times(T/Z'(1)) &\hookleftarrow &B\times(Z/Z')(1)\\[1mm]
   \longdownarrow & &\longdownarrow & &\longdownarrow  &&\longdownarrow &&\longdownarrow\\[1mm]
 ( A^n \times A^{*s})^2 & \stackrel{d}{\hookleftarrow} &  A^n \times A^{*s} &\stackrel{i}{\hookleftarrow} &  B &= &B &= &B\rlap{.}\\
  \end{array}
  \end{equation*}
 \end{construction}

\smallskip

Recall that we denote $\UR (M)$ the unipotent radical of $\Galmot(M)$. By \cite[Thm 0.1]{B03} the Lie algebra $\Lie \UR (M)$ is the extension of $B$ by $Z(1)$ defined by the adjoint action of the Lie algebra $(B,Z(1), [\, , \,])$ over $B+Z(1)$.

\begin{lemma}\label{eq:DimUR}
Let $K\subset\CC$ be algebraically closed. One has ${\Gr}_{-1}^{\W}( \Lie\Galmot(M)) =B  \subseteq A^n\times {A^*}^s$, $W_{-2} ( \Lie\Galmot(M)) =  Z(1) \subseteq T$ and
$$\dim_{\QQ} \UR (M) = 2\dim B + \dim Z(1).$$
In particular,
$$\dim\Galmot(M) = \dim\Galmot(A) + 2\dim B + \dim Z(1),$$
\end{lemma}
\begin{proof}
The computation of the graded parts of $\Lie\Galmot(M)$ is done in \cite[Cor 3.10]{B03}. 

The image of the motivic Galois group $\Galmot(M)$ via the fibre functor ``Hodge realization" of the tannakian category $\langle M\rangle^\otimes$, is the Mumford-Tate group $\mathrm{MT}(M)$. Since $K$ is algebraically closed, by \cite[Theorem 1.2.1]{A19} these two group schemes coincide. We then write 
\begin{equation*}\begin{matrix} 
\hfill\dim_{\QQ} \UR (M) & = &\dim_\QQ \UR (\mathrm{MT}(M)) \hfill\\
& = &\dim_\QQ \Lie \UR (\mathrm{MT}(M)) \hfill\\
& = &\dim_\QQ H_1 (B, \QQ ) +\dim_\QQ H_1(Z(1), \QQ) \hfill\\
& = &2\dim B + \dim Z(1) .\hfill\\
\end{matrix}\end{equation*}
The second equality just combines the above and \eqref{eq:Dim}. \end{proof}

\begin{remark} By definition, the motivic Galois group $\Galmot(M)$ and the Mumford-Tate group $\mathrm{MT}(M)$ of a 1-motive $M$ are  the fundamental groups associated to $M$, considering $M$ as an object of two different tannakian categories: one is Nori's tannakian category of mixed motives and the other is the tannakian category of mixed Hodge structures. In \cite[Theorem 1.2.1]{A19} André proves that, via the fibre functor ``Hodge realization", these two group schemes $\Galmot(M)$ and $\mathrm{MT}(M)$ coincide.
\end{remark}

\subsection{Torsors}\label{torsors}
Let $K\subset\CC$ be algebraically closed, in particular we assume that the torii split. Recall that $G$ is an extension of $A$ by the torus $T_0\simeq\GG_m^s$ defined over $K$, which corresponds to points $Q_1, \dots, Q_s \in A^*(K)$. Consider $n$ points $R_1,\dots,R_n$ of $G(\CC)$ living over the points $P_1,\dots,P_n $ in $A(\CC)$. As usual we will denote $p_\ell$ and $q_k$ logarithms of $P_\ell$ and $Q_k$. As in Construction \ref{B-Z} (1), denote by $B$  the smallest abelian sub-variety $B$ of $A^n \times A^{*s} $ containing the point $(P_1,\dots,P_n,Q_1, \dots, Q_s )$. Restricting the second projection $ A^n \times A^{*s} \rightarrow  A^{*s} $ to $B$ we get a group morphism $  \pi: B \to A^{*s}$. The image of $\pi$ is the smallest abelian sub-variety of $ A^{*s} $ containing the point $Q=(Q_1, \dots, Q_s )$, that we denote
 \begin{equation}\label{Bv*}
 B_{v^*}.
 \end{equation}
Now $\pi: B \rightarrow B_{v^*}$ is a surjective group morphism. If we denote $B_O$ the fibre of $B$ above the neutral element $O$ of $B_{v^*}$, that is
\[B_O = B \cap (A^n \times \{O\}),\]
the abelian variety $B$ is in fact a  $B_O$-torsor over $ B_{v^*}$. Observe that the fibre $B_O$ is itself an abelian variety. The fibre $B_Q$ of $B$ above the point $Q=(Q_1, \dots, Q_s) \in B_{v^*} $, namely
\begin{equation}\label{BQ}
B_Q = B \cap (A^n \times \{Q\}),
\end{equation}
is an algebraic variety but not necessarily a group variety. However it is the translate of the abelian variety $B_O$ by the point $(P,Q)=(P_1,\dots,P_n,Q_1,\dots,Q_s)$:
\begin{align*}
B_Q &= B \cap (A^n \times \{Q\}) \\
&=\big(B+(P,Q)\big)\cap (A^n \times \{Q\}) \\
&=(B\cap (A^n  \times \{O\})) + (P,Q) \\
&= B_O+(P,Q).
\end{align*}
Identifying $A^n \times \{Q\}$ with $A^n$ we get 
\begin{equation}\label{eq:B_QB_O}
B_Q = B_O+P.
\end{equation} 
 Since $B$ is endowed with a structure of $B_O$-torsor over $ B_{v^*}$, we have
 \begin{equation}\label{eq:dimB}
 \dim B = \dim B_{v^*} + \dim B_O = \dim B_{v^*} + \dim B_Q.
 \end{equation}
 
\smallskip
 
\begin{remark}\label{rk:dimB_Q}
	 Let $F=\End( A) \otimes_\ZZ \QQ$ be the (possibly skew) field of endomorphisms of a simple, $g$-dimensional abelian variety $A$. Then:
	\begin{itemize}
	\item  the dimension of $B$ is $g$ times the dimension of the $F$-vector space generated modulo $\sum_{j=1}^{2g}F(\omega_{1j},\dots,\omega_{gj})$ by the points $p_\ell,q_k$, where  $\ell=1,\dots,n$, $k=1,\dots,s$. Similarly, the dimension of $B_{v^*}$ is $g$ times the dimension of the $F$-vector space generated modulo $\sum_{j=1}^{2g} F(\omega_{1j},\dots,\omega_{gj})$ by the $q_k$'s. Then, the dimension of $B_O$ (and $B_Q$) is $g$ times the dimension of the $F$-vector space generated by the images of the $p_\ell$'s modulo $\sum_{j=1}^{2g} F(\omega_{1j},\dots,\omega_{gj})+\sum_{k=1}^sFq_k$.
   \item if we set $\HH_F = \HH_1(A^\natural(\CC),\ZZ)\otimes_\ZZ F$ \eqref{eq:H_1Anatural},
	the dimension of $B$ is also equal to $g$ times the dimension of the $F$-vector space generated by the generalized abelian logarithms $\tlog_{A}(P_1),\dots,\tlog_{A}(P_n)$, $\tlog_{A^*}(Q_1),\dots,\tlog_{A^*}(Q_s)$ modulo $\HH_F$.
		\item  if the $n+s$ points $P_1,\dots,P_n,Q_1, \dots, Q_s $ are $F$-linearly independent, then $\dim B= (n+s)g$, $\dim B_{v^*} = sg$, $\dim B_Q =ng$ and in particular $\dim B_Q$ is maximal. If the points $Q_1, \dots, Q_s $ are $F$-linearly independent and the points $P_1,\dots,P_n$ belong to the $F$-vector space generated by  $Q_1, \dots, Q_s $, then $\dim B= \dim  B_{v^*} = sg$, $\dim B_Q =0$ and in particular $\dim B_Q$ is minimal.
\end{itemize}
\end{remark}

Recall that we assume $K$ algebraically closed so that the torus $(X^\vee\otimes Y)(1)$ is split isomorphic to $\GG_m^{ns}$. As in \eqref{R} and \eqref{eq:idB}, having the $n$ points $R_1, \dots, R_n$ of the extension $G$ is equivalent to having a point $\widetilde{R}$ in the fibre of the  $\GG_m^{ns}$-torsor $i^*d^* \mathcal{B}$ above $(P,Q) \in B$. 
 By (\ref{BQ}) we can consider the pull-back via the inclusion $I_Q^1: B_Q \hookrightarrow B$ of the $\GG_m^{ns}$-torsor $i^*d^* \mathcal{B}$ over $B$: 
 \begin{equation}\label{eq:E'}
 E':={I_Q^1}^*i^*d^* \mathcal{B},
 \end{equation}
which is  a $\GG_m^{ns}$-torsor over $B_Q$. The inclusion $Z \hookrightarrow \ZZ^{ns}$ of modules induces the projection $Pr:\GG_m^{ns} \twoheadrightarrow Z(1)$. Denote by 
  \begin{equation} \label{eq:E}
E:= {Pr}_*E'= {Pr}_*{I_Q^1}^*i^*d^* \mathcal{B}
  \end{equation}
the push-down of the $\GG_m^{ns}$-torsor $E'$ via $Pr$. We have that $E$ is a $Z(1)$-torsor over $B_Q$. We still denote by $\widetilde{R}$ the point of $E$ living over $P\in B_Q$. We can summarize the situation in the following diagram:
   \[
  \begin{array}{ccccccccc}
  \GG_m^{ns} &= &  \GG_m^{ns} &= &\GG_m^{ns} &= &\GG_m^{ns} &\stackrel{Pr}{\longrightarrow}&Z(1) \\[1mm]
  \circlearrowright &     & \circlearrowright  &&  \circlearrowright &&\circlearrowright &&\circlearrowright\\[1mm]
  \mathcal{B} &  \hookleftarrow & d^*\mathcal{B}  & \hookleftarrow  & i^*d^*\mathcal{B} &  \hookleftarrow &E' &  \longrightarrow &E\\
   \longdownarrow & &\longdownarrow & &\longdownarrow  &&\longdownarrow  &&\longdownarrow \\
 ( A^n \times A^{*s})^2 & \stackrel{d}{\hookleftarrow} &  A^n \times A^{*s} &\stackrel{i}{\hookleftarrow} &  B &\stackrel{I_Q^1}{\hookleftarrow}& B_Q &= &  B_Q.\\
  \end{array}
  \]
  
  We have then
  
  \begin{lemma}\label{lem:tildeR}
  	Having the $n$-tuple $R=(R_1, \dots, R_n)$ of points of the extension $G$ is equivalent to having a point $\widetilde{R}$ in the fibre of the  $Z(1)$-torsor $E$ above the point $P$ of $ B_Q$. By construction this $Z(1)$-torsor   $E$  over $B_Q$ is the smallest torsor (with respect to pull-backs and push-downs) defined over $K$ and
containing the point $\widetilde R$.
  \end{lemma}

  \begin{proof}
  The abelian sub-varieties $B\subset A^n\times {A^*}^s$ and $B_{v^*}\subset {A^*}^s$ are defined over $K$. Since $Q \in {A^*}^s(K)$, we have that $B_Q$ is defined over $K$. Hence the torsor $E$ is defined over $K$.
  
  Then $B_Q$ is a $B_O$-torsor, it is the smallest torsor defined over $K$ and containing $P$. Furthermore, $Z(1)$ is the smallest sub-torus of $\GG_m^{ns}$ such that the torsor $E$ over $B_Q$, contains $\tilde R$ and is defined over $K$. Therefore, $E$ is the smallest such torsor with respect to pull-backs and push-downs.
  \end{proof}

\medskip

We now want to make a link between the $Z(1)$-torsor $E$ over $B$ and the semi-abelian variety $G$. Since the biextension $\mathcal{B}$ is constructed using $ns$ copies of the Poincar\'e biextension $\mathcal{P}$ of $(A,A^*)$ by $\GG_m$ and $ns$ copies of the Poincar\'e biextension $\mathcal{P}^*$ of $(A^*,A)$ by $\GG_m$, the pull-back $i_Q^*d^*\mathcal{B}$ of the  $\GG_m^{ns}$-torsor $d^*\mathcal{B}$ over $A^n \times A^{*s}$ via the inclusion $i_Q:  A^n \times \{Q\}  \hookrightarrow  A^n \times A^{*s}$ is in fact an extension of $A^n$ by $\GG_m^{ns}$. More precisely, if $v^*:\ZZ^s \rightarrow A^*$ (\emph{resp.} $V^*:\ZZ\to {A^*}^s$ as in \eqref{defV-V*}) is the group morphism defined by $Q=(Q_1, \dots, Q_s)$ which characterizes the extension $G$ of $A$ by $\GG_m^s$, then the extension $i_Q^*d^*\mathcal{B}$ of $A^n$ by $\GG_m^{ns}$ corresponds to the group morphism $(v^*)^n: (\ZZ^{s})^n \rightarrow A^{*n}$ (\emph{resp.} $(V^*)^n:\ZZ^n\to {A^*}^{ns}$) defined by the $n$-tuple $(Q, \dots, Q)$. In particular we have the following equivalent commutative diagrams 
   \[
  \begin{array}{ccc}
  \hfill\ZZ^{s} &\stackrel{v^*}{\longrightarrow}&  A^*\hfill\\[1mm]
  \hfill{\scriptstyle d_{\ZZ^{s}}} \longdownarrow\ \ &     & \longdownarrow{\scriptstyle d_{A^*}} \hfill\\
 \hfill(\ZZ^{s})^n & \stackrel{(v^*)^n}{\longrightarrow} &{A^*}^{n}\hfill\\
  \end{array}
  \qquad \qquad
  \begin{array}{ccc}
  \hfill\ZZ &\stackrel{V^*}{\longrightarrow}&  A^{*s}\hfill\\[1mm]
  \hfill{\scriptstyle d_\ZZ}\longdownarrow  &     & \longdownarrow{\scriptstyle(d_{A^*})^s} \\
  \hfill\ZZ^{n} & \stackrel{(V^*)^n}{\longrightarrow} &{A^*}^{ns}\hfill\\
  \end{array}
  \]
where we have denoted by $d_{\ZZ^{s}}, d_{\ZZ}, d_{A^*}$ the diagonal morphisms. Therefore the extension $i_Q^*d^*\mathcal{B}$ of $A^n$ by $\GG_m^{ns}$ is just $G^n:$
   \[i_Q^*d^*\mathcal{B} = G^n .\]
By (\ref{BQ}) we have an inclusion $I_Q^2: B_Q \hookrightarrow A^n \times \{Q\}\simeq A^n$ and we can consider the pull-back of the extension $G^n$ via this inclusion. This pull-back ${I_Q^2}^*G^n$ is  a $\GG_m^{ns}$-torsor over $B_Q$. Since the diagram of inclusions 
\[\begin{matrix}
\hfill B_Q &\stackrel{I^1_Q}{\hookrightarrow} &B\hfill\\[1mm]
\hfill\scriptstyle I^2_Q\ \hookdownarrow\ \ &&\ \hookdownarrow\scriptstyle\ i\hfill\\
\hfill A^n\times\{Q\} &\stackrel{i_Q}{\hookrightarrow} &A^n\times{A^*}^s\hfill\\
\end{matrix}\]
is commutative,
we have
  \[ {I_Q^2}^* G^n = {I_Q^2}^*i^*_Qd^* \mathcal{B} =   {I_Q^1}^*i^*d^* \mathcal{B} =E' .\]  
If we denote by $In:Z(1) \hookrightarrow \GG_m^{ns}$ the inclusion such that the composition $Pr \circ In$  with the projection $Pr:\GG_m^{ns} \twoheadrightarrow Z(1)$ is the identity, we can summarize the situation with a diagram:
   \begin{equation}\label{eq:E'Gn}
  \begin{array}{ccccccccc}
  \GG_m^{ns} &= &\GG_m^{ns} &= &(\GG_m^{s})^n &= &\GG_m^{ns} &\stackrel{In}{\hookleftarrow} &Z(1)\\[1mm]
  \circlearrowright & &\circlearrowright & &\hookdownarrow & &\circlearrowright & &\circlearrowright\\[1mm]
  \mathcal{B} &\hookleftarrow &d^*\mathcal{B} &\hookleftarrow  &G^n &\hookleftarrow &{I_Q^2}^* G^n=E' &  &E\\
  \longdownarrow & &\longdownarrow & &\longdownarrow & &\longdownarrow & &\longdownarrow\\
  (A^n\times A^{*s})^2
  &\stackrel{d}{\hookleftarrow} &  A^n\times A^{*s} &\stackrel{i_Q}{\hookleftarrow} &A^n\times\{Q\} &\stackrel{I_Q^2}{\hookleftarrow} &B_Q &= &B_Q.\\
  \end{array}
  \end{equation}
  
  \begin{lemma}\label{lem:DansGn}
  The torsor $E$ is embedded inside $G^n$.
  \end{lemma}
  
\begin{proof}
  Denote by $U$ a complement to the sub-module $Z$ of $\ZZ^{ns}$. In particular $\GG_m^{ns} = Z(1) \times U(1)$. The morphism $Pr:\GG_m^{ns} \twoheadrightarrow Z(1)$ is the projection to the first factor. By \cite[Expos\'e VII, 1.4, (1.4.4)]{SGA7} $\HH^1(B_Q,\GG_m^{ns} ) \cong \Ext^1( \ZZ[B_Q],\GG_m^{ns})$ and since $\Ext^1( \ZZ[B_Q],\GG_m^{ns})$ is additive in the variable $\GG_m^{ns}$, we have that $\HH^1(B_Q,\GG_m^{ns} ) \cong \HH^1(B_Q, Z(1) ) \times \HH^1(B_Q, U(1) ).$ In particular, the $\GG_m^{ns}$-torsor $E'$ is isomorphic to the torsor  $In_*E \times E''$, where $E''$ is the push-down (via the inclusion $ U(1) \hookrightarrow \GG_m^{ns}$) of a $U(1)$-torsor. By the diagram \eqref{eq:E'Gn} we have an inclusion of the $\GG_m^{ns}$-torsor $E'$ inside $G^n$. Therefore we get inclusions of torsors $E \hookrightarrow In_*E  \hookrightarrow E'  \hookrightarrow G^n$. 
\end{proof}

Via the inclusion of $E$ in $G^n$ furnished by the above lemma, the points $\tilde R$ and $R=(R_1,\dots,R_n)$ can be identified. Denote by $\tE$ the pull-back via the translation by $R$, $\Tr_R: G^n \to G^n$, $S\mapsto S+R$, of the $Z(1)$-torsor $E$ over $B_Q$:
  \begin{equation}\label{def:tE}
  \tE := \Tr_R^* E = E-R.
  \end{equation}
 It is an extension of the abelian variety $B_O$ by the torus $Z(1)$. 
  
\begin{remark}\label{rk:defEtE}
   Since $B$ is defined over $K$, the fibre $B_O$ is defined over $K$ and so is the extension $\tE$.
\end{remark}

 
 \section{Semi-abelian analogues of Schanuel conjecture}\label{semiabelianSchanuel}
 
 We first write down various special cases of the Generalized Period conjecture applied to semi-abelian varieties, which extend the famous Schanuel conjecture. Then we consider relative versions of these conjectures above the fields generated by the periods of the semi-abelian varieties. We dub these as weakening of the previous statements. In this section $K$ is an algebraically closed sub-field of $\CC$.

\subsection{Extended Schanuel conjectures.}

Consider the well known Schanuel conjecture:

\begin{conjecture}[Schanuel conjecture] \label{SC}
If $x_1,\ldots , x_n$ are complex numbers which are $\QQ$-linearly independent, then  
\[
\mathrm{tran.deg}_{\QQ}\, \QQ  \big(x_1,\ldots , x_n, e^{x_1}, \ldots , e^{x_n} \big)\geq n.
\]
\end{conjecture}

In \cite[Cor 1.3 and \S 3]{B02} the first author showed that the Schanuel conjecture {\it{is equivalent}} to the Generalized Period conjecture \ref{eq:GPC} applied to a 1-motive defined over $K = \overline{\QQ (e^{x_1}, \ldots , e^{x_n})}$ without abelian part, that is $[u:{\ZZ}^n \to \GG_m]$.

\medskip

The elliptic analogue of Schanuel conjecture {\it{is}}
the Generalized Period conjecture \ref{eq:GPC} applied to the 1-motive $[u:{\ZZ}^n \to E]$ defined over $K$ and with $E$ an elliptic curve, see \cite[Thm 1.2]{B02}.

\begin{conjecture}[Elliptic analogue of Schanuel conjecture]
Let $E$ be an elliptic curve, defined over $\CC$, and with field of endomorphisms $ F=\End(E) \otimes_\ZZ \QQ.$ Let $\omega_{1},\omega_{2},\eta_{1},\eta_{2}$ be the periods and quasi-periods of $E$, respectively. Let $g_{2}=60 \; \mathrm{G}_{4}$ and $g_{3}=140 \; \mathrm{G}_{6}$, where $\mathrm{G}_{4}$ and $\mathrm{G}_{6}$ are the Eisenstein series of weight 4 and 6. Let $\wp(z)$ and $\zeta(z)$ be the Weierstrass functions relative to the lattice $ \HH_1(E(\CC),\ZZ)$). If $p_{1},\dots,p_{n}$ are $n$ points of $\Lie E_\CC/(F\omega_1+F\omega_2)$ which are $F$-linearly independent, then the transcendence degree over $\mathbb Q$ of the field 
\[
\QQ \big(g_{2},g_{3}, \wp(p_1), \dots,\wp(p_{n}),\omega_{1},\omega_{2},\eta_{1}, \eta_{2},  p_1, \dots,p_{n},\zeta(p_1), \dots, \zeta(p_{n})\big)\]
is at least $2 n + \frac{4}{\dim_{\QQ} F}$.
\end{conjecture}
The only known case of this conjecture is proved by Chudnovsky in \cite[Chap.7, Cor.5.4]{Ch}, for $n=0$ and $E$ with complex multiplications.

\medskip
 
 The abelian analogue of Schanuel conjecture {\it{is}} the Generalized Period conjecture \ref{eq:GPC} applied to a 1-motive defined over $K$ without toric part, that is $[u:{\ZZ}^n \to A]$ with $A$ an abelian variety, \emph{see} \cite{V18}. Before we state it, we establish a lemma computing the dimension of an abelian variety in terms of dimensions of vector spaces.

Let $A$ be an abelian variety, defined over $K$ and of dimension $g$. Poincar\'e's complete reducibility theorem \cite[Chap.IV, \S19, Thm.1]{M74} allows to write $A\sim A_1^{r_1}\times \dots \times  A_m^{r_m}$ up to isogeny, with the $A_i$'s simple, defined over $K$ and pairwise non-isogenous. Denote $g_i=\dim A_i$, $F_i=\End(A_i)\otimes_\ZZ\QQ$, $\pi_{ij}:A\to A_i$, $i=1,\dots,m$, $j=1,\dots,r_i$, the projections of $A$ on the $j$-th factor $A_i$ of $A_i^{r_i}$ and $\pi_i=(\pi_{i1},\dots,\pi_{ir_i})$ the projection of $A$ on $A_i^{r_i}$.

\begin{lemma}\label{lem:Poincarereducibility}
With these notations, the dimension of the smallest abelian subvariety of $A^n$ containing a given point $P=(P_1,\dots,P_n)$ of $ A^n(\CC) $ is $\sum_{i=1}^{m}n_ig_i$, where $n_i$ is the dimension of the $F_i$-vector space generated in $A_i$ by the points $\pi_{i1}(P_1),\dots,\pi_{ir_i}(P_1),\dots,\pi_{i1}(P_n),\dots,\\ \pi_{ir_i}(P_n)$.
\end{lemma}
%
%
\begin{proof}
The smallest abelian sub-variety $B$ of $A^n$ containing $P$ is isogenous to $B_1\times\dots\times B_m$, where $B_i$ is the smallest abelian sub-variety of $A_i^{r_in}$ containing the point $(\pi_{i}(P_1),\dots,\pi_{i}(P_n))$. Then $B_i$ is isogenous to $A_i^{n_i}$ with $n_i$ the dimension of the $F_i$-vector space generated in $A_i$ by the points $\pi_{i1}(P_1),\dots,\pi_{ir_i}(P_1),\dots,\pi_{i1}(P_n),\dots,\pi_{ir_i}(P_n)$, because there is a bijection between equations of $B_i$ and $F_i$-linear relations between those points. Thus, the dimension of the variety $B$ is $\dim B = \sum_{i=1}^m\dim B_i = \sum_{i=1}^mn_ig_i$.
\end{proof}

\begin{conjecture}[Abelian analogue of Schanuel conjecture]\label{A}
Let $A$ be an abelian variety, defined over $K$ and of dimension $g$. Denote by $\Omega_A$ its matrix of periods. If $P_1,\dots,P_n$ are $n$ points of $A(K)$, then with notations as above we have that
\[\mathrm{tran.deg}_{\QQ}\, K \big(\Omega_A,  \tlog_A(P_1), \dots, \tlog_A(P_n)\big) \geq  2\sum_{i=1}^mn_ig_i +\dim{\Galmot}(A).\]
\end{conjecture}

If $L$ is the field of definition of the abelian variety $A$, the above  estimation is the strongest for $K=\overline{L(P_1, \dots , P_n)}$. Note that, if the points $\pi_{i1}(P_1),\dots,\pi_{ir_i}(P_1),\dots,\pi_{i1}(P_n),\dots,\pi_{ir_i}(P_n)$ are $F_i$-linearly independent in $A_i$, the dimension $n_i$ is maximal and equal to $nr_i$. If this occurs for all $i=1,\dots,m$, then $\sum_{i=1}^mn_ig_i = ng$.

\medskip

 The semi-abelian analogue of Schanuel conjecture  {\it{is}} the Generalized Period conjecture \ref{eq:GPC} applied to the 1-motive $M=[u:{\ZZ}^n \to G]$ defined over $K$ and with $G$ an extension of an abelian variety $A$ by a torus $\GG_m^s$.  Before we state it, we need to compute the dimensions of the 1-motives introduced in Construction \ref{B-Z} (1) and (3).

 According to Lemma \ref{lem:Poincarereducibility} applied to the point $(P,Q)$  of $(A^n \times A^{*s})(\CC),$ with $A$ and $A^*$ isogeneous to $ A_1^{r_1}\times\dots\times A_m^{r_m}$, the dimension 
 of the smallest sub-abelian variety $B$ of $A^n \times A^{*s}$ containing $(P,Q)$ is equal to
\begin{equation}\label{eq:dimbmax}
\dim B = \sum_{i=1}^md_ig_i,
\end{equation}
where $d_i$ is the dimension of the $F_i$-vector space generated in $A_i$ by the points $\pi_{ij}(P_\ell)$ and $\pi_{ij}^*(Q_k)$, $j=1,\dots,r_i$, $\ell=1,\dots,n$ and $k=1,\dots,s$, with the notations $\pi_{ij}$ introduced before Lemma \ref{lem:Poincarereducibility} and  
$\pi^*_{ij}$ the analogue for $A^*$. The maximal value of $d_i$ is $(n+s)r_i$, obtained when all the points $\pi_{ij}(P_\ell)$ and $\pi^*_{ij}(Q_k)$ are assumed $F_i$-linearly independent. When this occurs for $i=1,\dots,m$, then $\dim B = \sum_{i=1}^m (n+s) r_ig_i = (n+s)g$ (see Remark \ref{rk:dimB_Q} for other explicit computations of the dimension of B).

The dimension of the torus $Z(1)$ in Construction \ref{B-Z} (3) is denoted 
\begin{equation}\label{eq:dimz}
\dim Z(1) = t.
\end{equation}
The maximal value of $t$ is clearly $ns$.

\begin{conjecture}[Semi-abelian analogue of Schanuel conjecture] \label{SA}
Let $A$ be an abelian variety, defined over $K$ and of dimension $g$. 
Denote by $G$ the extension of $A$ by the torus $\GG_m^s$ which corresponds to $s$ points $Q_1, \dots, Q_s $ of $ A^*(K)$, and denote by $\Omega_G$ its matrix of periods. Let $R_1,\dots,R_n$ be $n$ points of $G(K)$ and $P_1,\dots,P_n$ their projections on $A(K)$.
With notations as above we have
\[
\mathrm{tran.deg}_{\QQ}\, K \big(\Omega_G,  \tlog_G(R_1), \dots , \tlog_G(R_n) \big) \geq  2\sum_{i=1}^md_ig_i  + t + \dim{\Galmot}(A).\]
\end{conjecture}

If $L$ is the field of definition of the abelian variety $A$, the above  estimation is the strongest for $K=\overline{L(Q_1, \dots, Q_s,R_1, \dots , R_n)}$.

Using the notation of the above conjecture, consider the 1-motive $M=[u:{\ZZ}^n \to G]$ defined over $K$, where the group morphism $u$ is defined by the $n$ points $R_1, ..., R_n$ of $G(K)$.

\begin{proposition}
Conjecture \ref{SA} is the Generalized Period conjecture \ref{eq:GPC} applied to the 1-motive $M$.
\end{proposition}

\begin{proof} By \eqref{eq:OmegaM} the periods of $ M$ are the elements of $\Omega_G$ and the generalized semi-abelian logarithms $\tlog_G(R_1), \dots , \tlog_G(R_n)$.
	Applying the Generalized Period conjecture \ref{eq:GPC} to the 1-motive $M$, we get that 
	$$\mathrm{tran.deg}_{\QQ}\, K\big(\Omega_G, \tlog_G(R_1), \dots , \tlog_G(R_n) \big) \geq \dim {\Galmot}(M). $$
	Since 
	by Lemma \ref{eq:DimUR} $ \dim \Galmot(M) = 2 \dim B+ \dim Z(1) + \dim \Galmot(A),$
	the equalities \eqref{eq:dimbmax} and \eqref{eq:dimz} allow then to conclude that Conjecture \ref{SA} and Conjecture \ref{eq:GPC} applied to $M$ coincide.
\end{proof}

\medskip

\subsection{Relative Semi-abelian conjectures.}

We now turn to relative versions of the above conjectures that we will need for our main Theorem \ref{MainThm}.
We use the terminology ``relative conjecture" because we are considering the relative transcendence degree of a field over a sub-field.

 Recall that the notations $Z(1)$ and $B_Q$ have been introduced in Construction \ref{B-Z} and \S\ref{torsors}. 

 \begin{conjecture}[Relative Semi-abelian conjecture]\label{WSA-V2}
Let $A$ be an abelian variety defined over $\oQQ$ and of dimension $g$. Let $G$ be an extension of $A$ by the torus $\GG_m^s$ which corresponds to $s$ points $Q_1, \dots, Q_s $ of $ A^*(\oQQ)$, and denote by $\Omega_G$ its matrix of periods. Let $R_1,\dots,R_n $ be $n$ points of $G(\CC)$. Then 
	\[\mathrm{tran.deg}_{\QQ (\Omega_G)}\, \QQ \big(\Omega_G, R_1,\dots,R_n, \tlog_G(R_1), \dots , \tlog_G(R_n) \big) \geq 2 \dim B_Q + \dim Z(1) .\]
\end{conjecture}

In the context of Conjecture \ref{SA}, here we have taken $K= \overline{\QQ (R_1,\dots,R_n)}.$

\begin{remark}
	If $s=0$, and so in particular $R_i=P_i \in A(K)$ for $i=1, \dots n$, under the hypothesis of Conjecture \ref{WSA-V2} the above inequality becomes
	\begin{equation} \label{WA}
	\mathrm{tran.deg}_{\QQ (\Omega_A)}\, \QQ \big(\Omega_A, P_1,\dots,P_n, \tlog_A(P_1), \dots , \tlog_A(P_n) \big) \geq   2 \dim B_Q
	\end{equation}
	which is the Weak Abelian Schanuel conjecture so called in \cite{V18} and \cite[Conjecture 3]{PSS}.
\end{remark}

A weak version of the Relative Semi-abelian conjecture \ref{WSA-V2} is 

\begin{conjecture}[Weak Relative Semi-abelian conjecture]\label{WRSA-V2}
Let $G$ be a semi-abelian variety defined over $\oQQ$ and denote by $\Omega_G$ its matrix of periods. Let $r$ be a point of $\Lie G_\CC$, then the transcendence degree over ${\QQ(\Omega_G)}$ of the field generated by $r$ and $\exp_G(r)$ is at least the dimension of the smallest translate in $G$ of an algebraic group defined over $\oQQ$ and which contains $\exp_G(r)$.
\end{conjecture}

To deduce this statement from Conjecture \ref{WSA-V2}, it suffices to observe that the transcendence degree over ${\mathbb Q(\Omega_G)}$ of the field generated by the integrals of the second kind is bounded above by $\dim B_Q$ and moreover $\dim B_Q+\dim Z(1)$ is the dimension of the torsor $E$ by Lemma \ref{lem:tildeR}.

\medskip

We now compute the dimension of $B_Q$, which is equal to $\dim B-\dim B_{v^*}$. With the same notations as in \S 3.1, applying Lemma \ref{lem:Poincarereducibility} to $B_{v^*}$ we get
\begin{equation}\label{eq:dimbv*}
\dim B_{v^*} = \sum_{i=1}^ms_ig_i 
\end{equation}
with $s_i$ the dimension of the $F_i$-vector space generated in $A_i$ by the points $\pi^*_{ij}(Q_k)$, $j=1,\dots,r_i$, $k=1,\dots,s$. Together with \eqref{eq:dimbmax}, this gives
\begin{equation}\label{eq:dimbq}
\dim B_Q = \sum_{i=1}^m(d_i-s_i)g_i  := \sum_{i=1}^m c_i g_i ,
\end{equation}
where $c_i=d_i-s_i$ turns out to be the dimension of the $F_i$-vector space generated in $A_i$ by the points $\pi_{ij}(P_\ell)$ and $\pi^*_{ij}(Q_k)$, $j=1,\dots,r_i$, $\ell=1,\dots,n$, $k=1,\dots,s$ modulo the one generated by the $\pi^*_{ij}(Q_k)$'s.

Recalling the definition of $t$ in \eqref{eq:dimz}, we get as a corollary of Conjecture \ref{WSA-V2}:

\begin{conjecture}[Explicit Relative Semi-abelian conjecture]\label{WSA}
With notations as above and as in Conjecture \ref{WSA-V2}, 
we have 
\[\mathrm{tran.deg}_{\QQ (\Omega_G)}\, \QQ \big(\Omega_G, R_1,\dots,R_n, \tlog_G(R_1), \dots , \tlog_G(R_n) \big) \geq 2\sum_{i=1}^m c_ig_i+t.\]
\end{conjecture}

\begin{remark}
If $g=0$ and $n=1$, and so in particular $R_1=R=(\tau_{1}, \dots, \tau_{s} ) \in \GG_m^s(K)$, the above inequality becomes
\[\mathrm{tran.deg}_{\QQ (2i \pi)}\, \QQ \big(2 i \pi,\tau_{1}, \dots, \tau_{s},\log(\tau_{1}),\dots,\log(\tau_{s}) \big) \geq t,\] where $t$ is the dimension of the $\QQ$-vector space generated by $\log(\tau_1),\dots,\log(\tau_s)$ in $\CC/2\mathrm i\pi\QQ$.
This is Schanuel conjecture (up to the transcendency of $\pi$).
\end{remark}

\begin{proposition}\label{lem:equivWSA-V2-GPC} The Relative Semi-abelian conjecture \ref{WSA-V2} and, in particular, the Explicit Relative Semi-abelian conjecture \ref{WSA} are consequences of the Generalized Period conjecture \ref{eq:GPC}.
\end{proposition}

\begin{proof} Let $x_1,\dots, x_n$ be generators of the character group ${\ZZ}^n$. Consider the 1-motive $M=[u:{\ZZ}^n \to G]$ with $u(x_\ell) = R_\ell$ for $\ell=1, \dots, n$. By \eqref{eq:OmegaM} the periods of $ M$ are the elements of $\Omega_G$ and the generalized semi-abelian logarithms $\tlog_G(R_1), \dots , \tlog_G(R_n)$. Applying the Generalized Period conjecture \ref{eq:GPC} to $M$ we get 
\[
\mathrm{tran.deg}_{\QQ}\, \QQ  \big(\Omega_G, R_1,\dots,R_n,  \tlog_G(R_1), \dots , \tlog_G(R_n) \big) \geq \dim \Galmot (M).
\]
Since the extension $G$ is defined over $\oQQ$, Grothendieck Period conjecture \ref{eq:CP} gives
\begin{equation}\label{dimGmot(G)}
\begin{aligned} 
\mathrm{tran.deg}_{\QQ}\,\QQ\big(\Omega_G \big) &= \dim\Galmot (G)\\
&= \dim \Galmot (A) + \dim \UR (G),
\end{aligned}\end{equation}
where the second equality is \eqref{eq:Dim} applied to the motive $[0\to G]$. Again by \eqref{eq:Dim}, 
 \[\dim \Galmot (M) = \dim \Galmot (A) + \dim \UR (M)\]
 and so
\begin{equation}\label{eq:DiffUR}
\begin{split}
\mathrm{tran.deg}_{\QQ (\Omega_G)}\, \QQ \big(\Omega_G, R_1,\dots,R_n,  \tlog_G(R_1), \dots , \tlog_G(R_n) \big) \geq \kern3cm\\
\dim \UR (M) - \dim \UR (G).
\end{split}\end{equation}
Applying Lemma \ref{eq:DimUR} to $M$, we get that $\dim \UR (M)= \dim Z(1) + 2 \dim B.$ Then applying Lemma \ref{eq:DimUR} to the Cartier dual of $G$, which is the 1-motive $M^*/\W_{-2}M^* =[v^* :Y^\vee \to A^*]$, we get that $ \dim \UR(G) = \dim \UR (M^* /\W_{-2}  M^*) = 2 \dim B_{v^*}$, where $B_{v^*}$ is defined in \eqref{Bv*} (observe that we have the first equality because $G$ and its Cartier dual define the same tannakian category, and so they have the same motivic Galois group). Finally the inequality (\ref{eq:DiffUR}) together with  the equalities  (\ref{eq:dimB}), \eqref{eq:dimbq} and \eqref{eq:dimz} give the conclusion.
\end{proof}
	
\begin{remark} In the context of the Explicit Relative Semi-abelian conjecture \ref{WSA}, assume $n=s=1$ and $P=Q$. Under these hypotheses the three 1-motives $[0 \to G]$, $M/\W_{-2}M=[v:\ZZ\to A ]$ and $M^*/\W_{-2} M^*=[v^*:\ZZ\to A^* ]$ are isogeneous. Hence the periods $\Omega_G$ of $[0 \to G]$, which are the periods $\Omega_A$ of $A$ and the $2g$ abelian integrals of the third kind  $ \int_{\tilde\gamma_j} \xi^G_{Q} $ for $j=1,\dots, 2g$, \textit{generate the same field over $\oQQ$} as the periods of $M/\W_{-2}M$, which are the periods $\Omega_A$ of $A$ and the $2g$ abelian integrals of the first and the second kind $\int_O^P \omega_i$, $\int_O^P \eta_i$ for $i=1,\dots,g$. 
Conjecture \ref{WSA-V2} then reduces to 
\[\mathrm{tran.deg}_{\QQ (\Omega_G)}\, \QQ \left(\Omega_G, R,  \int_O^R \xi^G_{Q} \right) \geq \dim Z(1)\]
with $\dim Z(1)=1$ or $0$.

\end{remark}


\section{Proof of the main theorem \ref{MainThm}}\label{proofMainThm}

Let $G$ be an extension of a $g$-dimensional abelian variety $A$ by a the torus $\GG_m^s$, which is defined over $\oQQ$ and which corresponds to the points $Q_1,\dots,Q_{s}$ of $A^*(\oQQ)$.

By convention, $\exp_G \tlog_G(R)= R$. Given a suitable set $S$, we denote
$\exp_G(S)$ (resp. $\log_G(S), \tlog_G(S))$ the set of elements $\exp_G(r)$ for $r \in S$ (resp. $\log_G(R), \tlog_G(R)$ for $R \in S$). By induction we can rewrite the two towers ${\cE}_n$ (\ref{eq:En}) and ${\cL}_n$ (\ref{eq:Ln}) in the following way 
\begin{align}
\nonumber &\cE_0 = \oQQ, &{\cE}_n  &= \overline{{\QQ}\Big(
 \exp_G\big(\Lie G_\CC(\cE_{n-1})\big)\Big)} \quad \mathrm{for} \; n \geq 1,\\
\nonumber &\cL_0=\oQQ, &{\cL}_n &= \overline{{\QQ}\Big( \Omega_G,  \; \widetilde{\log}_G\big(G({\cL}_{n-1})\big)\Big)} \quad \mathrm{for} \; n \geq 1.
\end{align}

We have to prove that if Relative Semi-abelian conjecture \ref{WSA-V2} is true, then  $ {\cE}= \cup_{n \geq 0} {\cE}_n$ and $ {\cL}=\cup_{n \geq 0} {\cL}_n $ are algebraically independent over $\oQQ.$ It is enough to prove that $ {\cE}_m$ and $ {\cL}$ are algebraically independent over $\oQQ$ for all $m$. We proceed by contradiction: if $\mathcal E_m$ and $\mathcal L$ are not algebraically independent, there exists $n$ such that $\mathcal E_m$ and $\mathcal L_n$ are not algebraically independent. We choose such a $n$ minimal. Clearly $m,n \geq 1.$
 
 As $\cE_m$ and $\cL_n$ are not algebraically independent over $\oQQ$,
there exist elements $\ell_1, \ldots, \ell_h$ of $\cL_n$, algebraically independent
over  $\oQQ$ which are algebraically dependent over $\cE_m$ i.e. there exists
a finite subset $\{e_1,\ldots,e_k\}$ of elements of $\cE_m$ 
such that $\ell_1,\ldots, \ell_h$ are algebraically dependent over $\oQQ(e_1, \ldots, e_k)$.

  Lemmas 2 and 3 of \cite{PSS} are still true if we replace abelian varieties with semi-abelian varieties since their proofs are based on the construction of the two towers $\{ {\cE}_n\}_n$ and $\{ {\cL}_n\}_n$. Applying \cite[Lemma 2]{PSS} to  $\{e_1,\dots , e_k\} \subseteq {\cE}_m ,$  there exists a finite subset $\mathcal{A} \subseteq {\cE}_{m -1}$ such that $\mathcal{A} \cup \{e_1,\dots , e_k\}$ is algebraic over the field $\oQQ \big(\exp_G (\mathcal{A}^{g+s}) \big) .$ On the other hand by \cite[Lemma 3]{PSS} applied to  $\{\ell_1,\ldots,\ell_h\} \subseteq {\cL}_n ,$ there exists a finite subset $\mathcal{C} \subseteq {\cL}_{n}^{g +s}$  with $\exp_G(\mathcal{C})\subseteq G( {\cL}_{n-1})$ such that the set $\{\mathrm{components\; of} \exp_G (\mathcal{C} ) \} \cup \{\ell_1,\ldots,\ell_h\}$ is algebraic over the field $\oQQ \big(\Omega_G, \mathcal{C} \big)$.
  
\medskip
  
Let $\vert \mathcal{A}^{g+s} \vert =n_1$ and $\vert \mathcal{C} \vert =n_2$. Consider the point 
   \[ \mathbf{r_1}:= (r_{11}, r_{12},\ldots, r_{1n_1})  \in \big(\Lie G_\CC\big)^{n_1} \cong \big( \Lie A_\CC \times \Lie \GG_{m,\CC}^{s} \big)^{n_1} \]
with $r_{1i} \in  \mathcal{A}^{g+s}$ for $i=1,\ldots,n_1$. Using $ \mathbf{r_1}$ we construct the 1-motive $M_1=[u_1: \ZZ \to G^{n_1}]$, where $ u_1(1)=\mathbf{R_1} = \exp_{G^{n_1}} ( \mathbf{r_1} ).$ Similarly, consider the point
   \[ \mathbf{r_2}:= (r_{21}, r_{22},\ldots,r_{2n_2}) \in \big(\Lie G_\CC \big)^{n_2} \cong  \big( \Lie A_\CC \times \Lie \GG_{m,\CC}^{s} \big)^{n_2} \]
with $r_{2i} \in  \mathcal{C}$ for $i=1, \ldots , n_2, $ and let $M_2$ be the 1-motive $[u_2: \ZZ \to G^{n_2}]$, where $u_2(1) = \mathbf{R_2}= \exp_{G^{n_2}} ( \mathbf{r_2} )$. Fix an extended logarithm $\widetilde{\log}_{G^{n_1}}(\mathbf{R_1})$ of the point $\mathbf{R_1}$. With these notations, we construct the two fields
 \begin{align}
 \nonumber K_1 &:=\overline{\QQ \big(\mathbf{R_1}, \tlog_{G^{n_1}} (\mathbf{R_1}) \big)},\\
 \nonumber K_2 &:=\overline{\QQ \big(\Omega_G, \mathbf{R_2}, \tlog_{G^{n_2}}(\mathbf{R_2}) \big)}.
 \end{align}

 Observe that the fields generated by $\Omega_G$, $\Omega_{G^{n_2}}$ and $\Omega_{G^{n_1}}$ are the same. By \cite[Lemma 2]{PSS}, $\{\mathrm{components\; of} \; \mathbf{r_1}\} \cup \{e_1,\dots , e_k\}$ is algebraic over  $\oQQ \big(\mathrm{components\; of} \; \mathbf{R_1} \big) $ and this implies the following two facts
 \begin{equation}\label{eq:e}
  \oQQ(e_1,\dots , e_k) \subseteq K_1 \enspace,
 \quad  K_1 =\overline{\QQ \big(\mathbf{R_1},\widetilde{\log}''_{G^{n_1}}(\mathbf{R_1}) \big)} \enspace,
 \end{equation}
where $\widetilde{\log}''_{G^{n_1}}(\mathbf{R_1})$ are the integrals of the second kind associated to the point $\mathbf{R_1}$. By \cite[Lemma 3]{PSS}, $\{\mathrm{components\; of}  \; \mathbf{R_2}\} \cup \{\ell_1,\ldots,\ell_h\}$ is algebraic over the field $\oQQ \big(\Omega_G,  \mathbf{r_2}  \big) $ and this implies that
  \begin{equation}\label{eq:l}
  \oQQ(\ell_1,\ldots,\ell_h) \subseteq K_2\enspace,
  \quad K_2 = \overline{\QQ \big(\Omega_G, \tlog_{G^{n_2}}(\mathbf{R_2}) \big)}\enspace.
  \end{equation}
   By Relative Semi-abelian conjecture \ref{WSA-V2} we have 
  \begin{align}
   \label{eq:K1prov}\mathrm{tran.deg}_{\QQ (\Omega_G)}\, K_1(\Omega_G ) \geq 
  \dim Z^1(1) + 2 \dim B_Q^1,\\
\label{eq:>K2}
  \mathrm{tran.deg}_{\QQ (\Omega_G)}\,K_2 \geq 
 \dim Z^2(1) + 2 \dim B_Q^2,
 \end{align}
where $Z^i(1),B_Q^i $ are the pure motives attached as in Construction \ref{B-Z} and \eqref{BQ} to the 1-motives $M_i$ for $i=1,2$. Since $\mathrm{tran.deg}_{\QQ(\Omega_G)}\, K_1(\Omega_G)\leq \mathrm{tran.deg}_{\QQ}\, K_1$, the inequality (\ref{eq:K1prov}) entails:
\begin{equation}\label{eq:>K1}
\mathrm{tran.deg}_{\QQ} K_1 
\geq \dim Z^1(1) + 2 \dim B_Q^1. 
\end{equation}

By \eqref{eq:E} and Lemma \ref{lem:tildeR} having the point $u_1(1)=\mathbf{R_1}$ of  $G^{n_1}$ is equivalent to having a point $\widetilde{\mathbf{R}}_1$ of the $Z^1(1)$-torsor $E_1$ over $B^1_Q$. Thanks to Lemma \ref{lem:tildeR}, the torsor $E_1$ is defined over $\oQQ$ and therefore we have
 \[\mathrm{tran.deg}_{\QQ }\, \oQQ (\mathbf{R_1}) = \mathrm{tran.deg}_{\QQ }\, \oQQ (\widetilde{\mathbf{R}}_1)\leq \dim E_1 =  \dim Z^1(1) +  \dim B_Q^1. \]
As in (\ref{eq:B_QB_O}) the algebraic variety $B_Q^1$ is the translation of the abelian variety $B_O^1$ whose dimension is an upper bound for the transcendence degree of the field generated by the abelian integrals of the second kind associated to its points. Therefore
\[\mathrm{tran.deg}_{\QQ } \,K_1  \leq  \dim Z^1(1) + 2 \dim B_Q^1.  \]
Because of (\ref{eq:>K1}) we can conclude 
\begin{equation}\label{eq:=K1}
\mathrm{tran.deg}_{\QQ} K_1 = \dim Z^1(1) + 2 \dim B_Q^1. 
\end{equation}

As in the case of $\mathbf{R_1} \in G^{n_1}$, by \eqref{eq:E} and Lemma \ref{lem:tildeR} having the point $u_2(1)=\mathbf{R_2}$ of  $G^{n_2}$ is equivalent to having a point $\widetilde{\mathbf{R}}_2$ of the $Z^2(1)$-torsor $E_2$ over $B^2_Q$. Since $\Lie E_{2,\CC} = \Lie B^2_{Q,\CC} \times \Lie Z^2(1)_\CC$ has the same dimension as $E_2$ and
$\mathbf{r_2} = \log_{G^{n_2}} (\mathbf{R_2})$ is an element of $\Lie E_{2,\CC}$, we have that
 \[\mathrm{tran.deg}_{\QQ (\Omega_G)}\, \oQQ (\Omega_G ,\log_{G^{n_2}} (\mathbf{R_2}))\leq \dim \Lie E_{2,\CC} = \dim E_{2} =  \dim Z^2(1) +  \dim B_Q^2. \]
 Again, as in \eqref{eq:B_QB_O}, the algebraic variety $B_Q^2$ is the translation of the abelian variety $B_O^2$ whose dimension is an upper bound for the transcendence degree of the field generated by the abelian integrals of the second kind associated to its points. Therefore
\[\mathrm{tran.deg}_{\QQ(\Omega_G) } \,K_2 \leq \dim Z^2(1) + 2 \dim B_Q^2.  \]
Remembering (\ref{eq:>K2}) we can conclude 
  \begin{equation} \label{eq:=K2}
\mathrm{tran.deg}_{\QQ (\Omega_G)}\,K_2 =
\dim Z^2(1) + 2 \dim B_Q^2 .
\end{equation}

Consider now the point 
\[ \mathbf{r_3}:= (\mathbf{r_1},\mathbf{r_2})  \in \big(\Lie G_\CC\big)^{n_1+n_2} \cong \big( \Lie A_\CC \times \Lie \GG_{m,\CC}^{s} \big)^{n_1+n_2}, \]
and construct the 1-motive $M_3=[u_3: \ZZ \to G^{n_1+n_2}]$, where $ u_3(1)= \mathbf{R_3}= \exp_{G^{n_1+n_2}} ( \mathbf{r_3} ) = (\mathbf{R_1},\mathbf{R_2})$.
By Conjecture \ref{WSA-V2}  we have that 
\[
\mathrm{tran.deg}_{\QQ (\Omega_G)}\, K_3 \geq 
\dim Z^3(1) + 2 \dim B_Q^3 ,
\]
where $ K_3 :=\overline{\QQ \big(\Omega_G, \mathbf{R_3}, \tlog_{G^{n_1+n_2}}(\mathbf{R_3}) \big)}$ and $Z^3(1),B_Q^3 $ are the pure motives attached as in Construction \ref{B-Z} and \eqref{BQ} to the 1-motive $M_3$. Observe that by (\ref{eq:e}) and (\ref{eq:l})
\[ K_3 :=\overline{\QQ \big(\Omega_G, \mathbf{R_1}, \tlog''_{G^{n_1}}(\mathbf{R_1}), \tlog_{G^{n_2}}(\mathbf{R_2}) \big)}. \]
Since by construction $K_3=K_1K_2,$ we get 
\begin{equation} \label{eq:>K3}
\mathrm{tran.deg}_{\QQ (\Omega_G)}\, K_1 K_2 \geq 
\dim Z^3(1) + 2 \dim B_Q^3.
\end{equation}
Now from Lemma \ref{lem:last} (that we will show at the end of this proof), equalities (\ref{eq:=K1}), (\ref{eq:=K2}) and inequality (\ref{eq:>K3}), we get 
\[
\mathrm{tran.deg}_{\QQ (\Omega_G)}\, K_1 K_2 \geq
\mathrm{tran.deg}_{\QQ} K_1  + \mathrm{tran.deg}_{\QQ(\Omega_G)} K_2. \]
Adding $\mathrm{tran.deg}_{\QQ} \; \QQ (\Omega_G)$ to both sides of this last inequality, we obtain 
\[
\mathrm{tran.deg}_{\QQ} K_1 K_2 \geq
\mathrm{tran.deg}_{\QQ} K_1  + \mathrm{tran.deg}_{\QQ} K_2, \]
that is 
\[
\mathrm{tran.deg}_{\QQ} K_1 K_2 =
\mathrm{tran.deg}_{\QQ} K_1  + \mathrm{tran.deg}_{\QQ} K_2. \]
Therefore the fields $K_1$ and $ K_2$ are algebraically independent over $\oQQ.$
But this is a contradiction to our assumption that $\ell_1,\ldots,\ell_h \in K_2$ (\ref{eq:l}) are algebraically dependent over  
$\oQQ(e_1,\dots , e_k) \subseteq K_1 $ (\ref{eq:e}). This finishes the proof of our Main Theorem \ref{MainThm}.

\begin{lemma}\label{lem:last}
	With the above notations 
	\[ 	\dim Z^3(1) + 2 \dim B_Q^3 =
	\dim Z^1(1) +	\dim Z^2(1) + 2 (\dim B_Q^1 + \dim B_Q^2) .\]
\end{lemma}

\begin{proof} By Lemma \ref{lem:tildeR} and \eqref{def:tE}, the $Z^i(1)$-torsor
		 $E_i$ over $B^i_Q$ is the smallest translate of a semi-abelian sub-variety $\tE_i$ of $G^{n_i}$, defined over $\oQQ$, which contains the point $\mathbf{R}_i$ (for $i=1,2,3$). Recall that ${\tE}_i$ is an extension of the abelian variety $B_O^i=B^i_Q-\pi({\bf R}_i)$ by the torus $Z^i(1)$ (here $\pi:G^{n_i} \to A^{n_i}$ is the projection underlying the semi-abelian variety $G^{n_i}$). In particular we have
	\begin{equation}\label{eq:dimE}
	\dim \tE_i =\dim E_i  = \dim Z^i(1) +\dim B^i_Q .
	\end{equation}
	By definition $\mathbf{R}_3=(\mathbf{R}_1,\mathbf{R}_2)\in E_1\times E_2$, which implies that
	\begin{equation}\label{inclE}
	E_3\subset E_1\times E_2 \qquad \mathrm{and} \qquad \widetilde E_3\subset\widetilde E_1\times\widetilde E_2.
	\end{equation}
	Since $B^i_Q=\pi(E_i)$, we have that $B^3_Q\subset B^1_Q\times B^2_Q$ and so
	\begin{equation}\label{ineqB}
	\dim B^3_Q \le\dim B^1_Q +\dim B^2_Q .
	\end{equation}
	Similarly, we have $Z^3(1)\subset Z^1(1)\times Z^2(1)$ and so
	\begin{equation}\label{ineqZ}
	\dim Z^3(1) \le \dim Z^1(1) + \dim Z^2(1) .
	\end{equation}
	In order to prove our Lemma, it suffices to check that 
	\begin{equation}\label{eq:dim123}
	\dim \tE_3 = \dim \tE_1 +\dim \tE_2
	\end{equation}
	 since, together with (\ref{eq:dimE}), (\ref{ineqB}) and (\ref{ineqZ}), it will imply
	\[\begin{array}{c}
	\dim B^3_Q =\dim B^1_Q +\dim B^2_Q ,\\[2mm]
	\dim Z^3(1) =\dim Z^1(1) +\dim Z^2(1) .
	\end{array}\]
As observed in Remark \ref{rk:defEtE} and Lemma \ref{lem:tildeR}, for $i=1,2,3$ both $E_i$ and $\tE_i = E_i - \mathbf{R}_i$ are defined over $\oQQ$ and thus $E_i(\oQQ)\not=\emptyset$. By (\ref{inclE})  the group sub-variety $\tE_3$ of the product $\tE_1 \times \tE_2$ fits into an exact sequence of group varieties defined over $\oQQ$
$$0 \rightarrow \tE_3 \hookrightarrow \tE_1 \times \tE_2 \xrightarrow{\lambda} \tE' \rightarrow 0,$$
where the group morphism $\lambda$ is $(\lambda_1,\lambda_2)$, with $\lambda_i \in \Hom(\tE_i,\tE') $ for $i=1,2$. In particular $\lambda(a,b)= \lambda_1(a) + \lambda_2(b)$. Since $\tE_3=\Tr^*_{\mathbf{R}_3}(E_3)=E_3-\mathbf{R}_3$, the morphism $\lambda\circ\Tr_{-\mathbf{R}_3}\vert_{E_1\times E_2} = (\lambda_1\circ\Tr_{-\mathbf{R}_1}\vert_{E_1},\lambda_2\circ\Tr_{-\mathbf{R}_2}\vert_{E_2}):E_1\times E_2 \to \tE'$ vanishes on $E_3$.

Fix a point $S=(S_1,S_2)\in E_3(\oQQ)$ which is a specialisation of $\mathbf{R}_3$. Since $S-\mathbf{R}_3$ belongs to $\tE_3$, the point $\lambda(S-\mathbf{R}_3)$ is zero in $\tE',$ i.e. $\lambda_1(S_1-\mathbf{R}_1)+\lambda_2(S_2-\mathbf{R}_2)=0$. Recall that $\mathcal{A} \subseteq {\cE}_{m -1}$ and so $\lambda_1(S_1-\mathbf{R}_1) \in \tE'(\cE_m)$. On the other hand, the subset $\mathcal{C} $ is such that  $\exp_G(\mathcal{C}) \subseteq G(\cL_{n-1})$ and so $\lambda_2(S_2-\mathbf{R}_2) \in \tE'(\cL_{n-1})$. From the choice of the pair $(m,n)$, $\cE_m$ and $ \cL_{n-1}$ are algebraically independent over $\oQQ$, and so $\cE_m \cap \cL_{n-1} = \oQQ$. Hence the point $\mathbf{T}=\lambda_1(S_1-\mathbf{R}_1)=-\lambda_2(S_2-\mathbf{R}_2) $ belongs to $ \tE'(\oQQ)$.

The equality $\mathbf{T}= \lambda_1(S_1-\mathbf{R}_1) $ is expressed by polynomial equations in some projective embeddings of $\tE'$ and $ \tE_1 \times \tE_2$. These equations hold for any specialization of $\mathbf{R}_1$. Thus $\mathbf{T}=\lambda_1(S_1-R_1)$ for any specialization $R_1$ of $\mathbf{R}_1$. In particular, for $R_1=S_1$ we get $\mathbf{T}=0$ and $\mathbf{R}_i\in S_i+\ker\lambda_i$ for $i=1,2$. 
 Since $E_i$ is the smallest translate of a semi-abelian variety defined over $\oQQ$ and containing $\mathbf{R}_i$, we deduce $E_i\subset S_i+\ker\lambda_i$. But then $\tE_i\subset \ker\lambda_i$ and thus $\tE_1\times\tE_2 \subset \ker\lambda=\tE_3$ for $i=1,2$. Together with \eqref{inclE} it follows that $\tE_3=\tE_1 \times \tE_2$ and in particular we get the expected formula (\ref{eq:dim123}).
\end{proof}


\section{Application}\label{application}

Let $E \simeq \CC/\Lambda$ be an elliptic curve defined over $\CC$ and $E^*$ its dual. We identify $E^*$ with $\CC/\Lambda^*,$ where $z^*\in \CC$ is viewed as a $\CC$-antilinear form $z\to z^*(z)=\overline{z}z^*$ on $\CC$ and
$$\Lambda^* = \{\lambda^*\in\CC; \; {\rm Im}(\overline{\Lambda}\lambda^*)\subset \ZZ\}.$$
Let $\omega_{1},\omega_{2},\eta_{1},
\eta_{2}$ be the periods and the quasi-periods of $E$, and let $\wp(z),\zeta (z)$ and $\sigma(z)$ be the Weierstrass functions relative to the lattice $\Lambda = \HH_1(E(\CC),\ZZ)$.

Write $\Lambda=\ZZ\omega_1+\ZZ \omega_2$ and $\Lambda^*=
\ZZ \omega^*_1+\ZZ \omega^*_2$ with $\omega_1,\omega^*_1 \in \RR$. Then ${\rm Im}(\overline{\omega}_i\Lambda^*)\subset \ZZ$, $i=1,2$, gives $\omega^*_1=\frac{\omega_1}{{\rm Im}(\omega_1\overline{\omega}_2)}$ and $\omega^*_2= - \frac{\omega_2}{{\rm Im}(\omega_1\overline{\omega}_2)}$. We identify the real Lie algebra $\Lie E_\RR$ of $E$ (\emph{resp.} $\Lie E^*_\RR$) with $\Lambda\otimes_\ZZ \RR $ (\emph{resp.} $\Lambda^*\otimes_\ZZ \RR$). For $z=\alpha_1{\omega}_1+\alpha_2{\omega}_2\in \Lie E_\RR $ and $z^*=\alpha^*_1\omega^*_1+\alpha^*_2\omega^*_2\in\Lie E^*_\RR,$ one computes easily
\begin{equation}\label{dualityproduct}
\alpha_1 = \frac{1}{\omega_1}\left({\rm Re}(z)-\frac{{\rm Re}(\omega_2)}{{\rm Im}(\omega_2)}{\rm Im}(z)\right)\quad\mbox{and}\quad \alpha_2 = \frac{1}{{\rm Im}(\omega_2)}{\rm Im}(z)
\end{equation}
and similarly for $\alpha^*_1$ and $\alpha^*_2$. Thus, the duality product is
\begin{equation}\label{dualityprodut}
\langle z,z^*\rangle := {\rm Im}(\overline{z}z^*) = {\rm Re}(z){\rm Im}(z^*) - {\rm Re}(z^*){\rm Im}(z) = \alpha^*_1\alpha_2 - \alpha^*_2\alpha_1.
\end{equation}

We may identify $E^*$ and $E$ via the map $\rho:z^*\mapsto z^*\mathrm{Im}(\omega_1\overline{\omega}_2)$.

In  Table \ref{tablecorr} we summarize the correspondence between the usual data for an abelian variety $A$ and an elliptic curve $E$.
\begin{table}[t]{\renewcommand{\arraystretch}{1.5}
    \caption{Analogies between an elliptic curve and an abelian variety}\label{tablecorr}
	\begin{tabular}{|c|c|}
    \hline
		elliptic curve $E$ & abelian variety $A$\\ 
	\hline
		1 & $g$\\
	\hline
		$\gamma_1,\gamma_2$ & $\gamma_1,\dots,\gamma_{2g}$\\
	\hline
		$\omega$ & $\omega_1,\dots,\omega_g$\\
	\hline
		$\eta$ & $\eta_1,\dots,\eta_g$\\
	\hline
		$[\wp(z):\wp'(z):1]$ & $\exp_A(z)$\\
	\hline
		$\zeta(z)$ & $h_1(z),\dots,h_g(z)$\\
	\hline
		$\sigma(z)$ & $\theta(z)$\\
	\hline
		$f_q(z) = \frac{\sigma(z+q)}{\sigma(z)\sigma(q)} e^{-\zeta(q) z }$ & $F_q(z)= \frac{\theta(z+q)}{\theta(z)\theta(q)}  e^{- \sum_{i=1}^g h_i(q) z_i }$\\[3pt]
	\hline
	\end{tabular}}
\end{table} 
which is the elliptic analogue of the function $F_q(z)$ defined in \eqref{eq:def-Fq}. 
 The logarithmic differential of this function is the pull-back via the exponential map of a differential of the third kind $\xi_Q$ on $E.$ Serre's function has quasi-quasi periods equal to the exponentials of the elliptic integrals of the third kind  $\int_{\gamma_i} \xi_Q = \eta_i q - \omega_i \zeta(q)$ for $i=1,2$ (see \cite[Appendice II by J.-P. Serre]{W} and \cite[(1.8)]{B19}).

Note that in the abelian case the function $\theta(z)$ is chosen so that  $\theta(0)=1$ and the factor of automorphy of the function $\frac{\theta (z+\rho(z^*))}{\theta(z)\theta(\rho(z^*))}$ coincides with the factor of automorphy of the Poincar\'e line bundle  $\cP$. In the elliptic case the function $\sigma(z)$ differs from this choice by a rational function on $E$ and satisfies $\sigma(0)=0, $ as we will see now.

The Poincar\'e biextension  $\cP$ of $(E, E^*)$ by $\GG_m$ is the line bundle defined as
$$\cP := (\Lie E_\CC \times \Lie E_\CC^* \times \GG_m)/(\Lambda\times\Lambda^*) $$
where the action of $(\lambda,\lambda^*)\in\Lambda\times\Lambda^*$ is given by 
$(z,z^*,t)\mapsto(z+\lambda,z^*+\lambda^*,t \cdot a(\lambda,\lambda^*,z,z^*))$,
with the factor of automorphy
\begin{equation}\label{autofactor}
a(\lambda,\lambda^*,z,z^*) = e^{\pi(\overline{\lambda}\lambda^*+z\overline{\lambda}^*+\overline{\lambda}z^*)} = e^{\pi (\lambda\overline{\lambda}^*+z\overline{\lambda}^*+\overline{\lambda}z^*)},
\end{equation}
since ${\rm Im}(\overline{\lambda}\lambda^*) \in \ZZ$, see \cite[\S2.5, pages 37-41]{BL}.

\medskip

\begin{lemma}\label{lem:eta}
	 For $z=\alpha_1\omega_1+\alpha_2\omega_2 \in \Lie E_\RR$, the number $\eta(z) = \alpha_1\eta_1+\alpha_2\eta_2$ reads
	\begin{equation}\label{eq:eta}
	\eta(z) = \frac{\eta_1}{\omega_1}z - \frac{2{\rm i}\pi}{{\rm Im}(\omega_1\overline{\omega}_2)}{\rm Im}(z).
	\end{equation}
	The function $\sigma(z)$ is a theta function with factor of automorphy 
	$$\psi(\lambda)e^{\eta(\lambda)(z+\frac{\lambda}{2})} = \psi(\lambda)e^{\left(\frac{\eta_1}{\omega_1}\lambda - \frac{\pi}{{\rm Im}(\omega_1\overline{\omega}_2)}(\lambda-\overline{\lambda})\right)(z+\frac{\lambda}{2})} = \psi(\lambda)e^{\frac{\pi(\overline{\lambda}+A\lambda)(z+\lambda/2)}{{\rm Im}(\omega_1\overline{\omega}_2)}}$$
	where $\pi A:=\eta_1{\rm Im}(\overline{\omega}_2)-\pi$ and $\psi(\lambda)=\begin{cases}+1&\text{if\ }\lambda\in2\Lambda\\-1&\text{otherwise\ }\end{cases}$.
\end{lemma}

\begin{proof}
	By the definition of $\eta(z)$ and \eqref{dualityproduct}, we have
	\begin{align*}
	\eta(z) &= \frac{\eta_1}{\omega_1}\left({\rm Re}(z)-\frac{{\rm Re}(\omega_2)}{{\rm Im}(\omega_2)}{\rm Im}(z)\right) + \frac{\eta_2}{{\rm Im}(\omega_2)}{\rm Im}(z)\\
	&= \frac{\eta_1}{\omega_1}{\rm Re}(z) + \frac{\eta_2\omega_1-\eta_1{\rm Re}(\omega_2)}{\omega_1{\rm Im}(\omega_2)}{\rm Im}(z)\\
	&= \frac{\eta_1}{\omega_1}z - \frac{\eta_2\omega_1-\eta_1\omega_2}{{\rm Im}(\omega_1\overline{\omega}_2)}{\rm Im}(z)\\
	\end{align*}
	which gives the result using Legendre's formula ({\it see} \cite[Chap.18, \S1, pp.241]{L87}). The factor of automorphy of the sigma function is given in \cite[Chap.18, \S1, Thm 1]{L87}).
\end{proof}

\medskip

\begin{corollary}\label{cor:theta-sigma}
	The theta function 
	\[  \tilde{\theta}(z) = \sigma(z)e^{\frac{-\pi Az^2}{2{\rm Im}(\omega_1\overline{\omega}_2)}}\]
	is such that  $\frac{\tilde{\theta}(z+\rho(z^*))}{\tilde{\theta}(z)\tilde{\theta}(\rho(z^*))}$ 
	is a meromorphic section of the pull-back $(\exp_{E} \times \exp_{ E^*})^* \cP$ of the Poincar\'e biextension on $\Lie E_\RR \times \Lie E^*_\RR$. Explicitly, this meromorphic section is given by
	\begin{align}
	\nonumber	 \Lie E_\CC \times \Lie E^*_\CC& \longrightarrow \GG_{m} \\
	\nonumber (z,z^*) & \longmapsto   \frac{\sigma(z+\rho(z^*))}{\sigma(z)\sigma(\rho(z^*))} e^{\frac{-\pi Az z^*}{{\rm Im}(\omega_1\overline{\omega}_2)}}.
	\end{align} 
\end{corollary} 

\begin{proof}
The factor of automorphy of the trivial theta function $e^{\frac{\pi Az^2}{2{\rm Im}(\omega_1\overline{\omega}_2)}}$ is $e^{\frac{\pi A\lambda(z+\lambda/2)}{{\rm Im}(\omega_1\overline{\omega}_2)}}$. Thus, by the above Lemma the factor of automorphy of the theta function $\tilde{\theta}(z)= \sigma(z)e^{\frac{-\pi Az^2}{2{\rm Im}(\omega_1\overline{\omega}_2)}}$ is \\ $\psi(\lambda)e^{\frac{\pi\overline{\lambda}(z+\lambda/2)}{{\rm Im}(\omega_1\overline{\omega}_2)}}.$
	
	The factor of automorphy of the theta function $\frac{\tilde{\theta}(z+\rho(z^*))}{\tilde{\theta}(z)\tilde{\theta}(\rho(z^*))}$ on $\CC^2$ relatively to $\Lambda\times\Lambda^*$ is
	$$\frac{\psi(\lambda+\rho(\lambda^*))}{\psi(\lambda)\psi(\rho(\lambda^*))}\times
	e^{\pi\left(\frac{(\overline{\lambda}+\rho(\overline{\lambda}^*))(z+\rho(z^*)+(\lambda+\rho(\lambda^*))/2)}{{\rm Im}(\omega_1\overline{\omega}_2)}-\frac{\overline{\lambda}(z+\lambda/2)}{{\rm Im}(\omega_1\overline{\omega}_2)}-\frac{\overline{\rho(\lambda}^*)(\rho(z^*)+\rho(\lambda^*)/2)}{{\rm Im}(\omega_1\overline{\omega}_2)}\right)},$$
	which is equal to
	$$e^{{\rm i}\pi{\rm Im}(\overline{\lambda}\lambda^*)}\times e^{\pi({\rm Re}(\overline{\lambda}\lambda^*)+\overline{\lambda}^*z+\overline{\lambda}z^*)}.$$
	We recognise the factor of automorphy \eqref{autofactor}, so that $\frac{\tilde{\theta}(z+\rho(z^*))}{\tilde{\theta}(z)\tilde{\theta}(\rho(z^*))}$ is a meromorphic section of the pull-back $(\exp_{E} \times \exp_{E^*})^* \cP$ of the Poincar\'e biextension on $\Lie E_\RR \times \Lie E^*_\RR$.
\end{proof}

Observe that the theta functions $\theta(z) $ and $\tilde{\theta}(z)$ have the same factor of automorphy and so their ratio $\frac{\theta(z)}{\tilde{\theta}(z)}$ is a rational function on $E$.

\subsection{Elliptic Weil pairing}

 In \S\ref{motivicGaloisgroups} we have introduced the biextension $\mathcal{B}$ constructed using one copy of the Poincar\'e biextension $\mathcal{P}$ of $(E,E^*)$ by $\GG_m$, one copy of the Poincar\'e biextension $\mathcal{P}^*$ of $(E^*,E)$ by $\GG_m$, one copy of the trivial biextension of $(E,E)$ by $\GG_m$ and one copy of the trivial biextension of $(E^*,E^*)$ by $\GG_m$. The pull-back $d^*\mathcal{B}$ via the diagonal morphism $d:E \times E^*\to(E \times E^*)^2$ is isomorphic to the line bundle $\cP\otimes\overline{\cP}^{-1}$, since $\cP^*= s^* \overline{\cP}^{-1}$ with $s: E^* \times E\to E\times E^*$ the morphism which exchange the factors.  
 
In particular by \eqref{autofactor} the factor of automorphy of $\cP \otimes \overline{\cP}^{-1}$ is 
\begin{equation}\label{eq:section}
a_0(\lambda,\lambda^*,z,z^*)=
a(\lambda,\lambda^*,z,z^*){\overline{a(\lambda,\lambda^*,z,z^*)}}^{-1} =  e^{2{\rm i}\pi{\rm Im}(z\overline{\lambda}^*+ \overline{\lambda} z^*)}.
\end{equation}

\medskip
According to \cite[\S 3.3]{B01}:

\begin{definition}\label{def-weilpairing}
	The Lie realization of the Weil pairing \eqref{eq:Weil-pairing} of the elliptic curve $E$ is
	$$\begin{matrix}
	\mathcal{W}_E: & \Lie E_\CC \times \Lie E^*_\CC & \longrightarrow 
	&\GG_m\hfill\\
	&(z,z^*) &\longmapsto &e^{2{\rm i}\pi{\rm Im}(\overline{z}z^*)}.
	\end{matrix}$$
\end{definition}

Extending the function $\eta(z)$ in \eqref{eq:eta} to $\CC$ by $\RR$-bilinearity, consider the following two variables function
\begin{align}
(\CC\setminus\Lambda)^2 & \longrightarrow  \CC\\
\nonumber	(z,q) & \longmapsto \tilde{f}_{q}(z) := \frac{\sigma(z+q)}{\sigma(z)\sigma(q)}e^{-\eta(q)z}.
\end{align}
Observe that the ratio $f_{q}(z)/\tilde f_{q}(z)$ is just the exponential function $e^{(\eta(q)-\zeta(q))z}$, which is a trivial theta function in $z$. The ratio $\tilde{f}_{q}(z)/\tilde{f}_{z}(q) = e^{\eta(z)q-\eta(q)z}$ is well defined on $\CC^2$. Moreover if $z=\alpha_1\omega_1+\alpha_2\omega_2\in\Lie E_{\RR}$, $z^*=\alpha^*_1\omega^*_1+\alpha^*_2\omega^*_2 \in \Lie E^*_\RR $, by Lemma \ref{lem:eta}, \eqref{dualityprodut} and Legendre's formula ({\it see} \cite[Chap.18, \S1, pp.241]{L87}), we have
\begin{equation}\label{eq:ratio}
\frac{\tilde{f}_{\rho(z^*)}(z)}{\tilde{f}_z(\rho(z^*))}
= e^{\eta(z)\rho(z^*)-\eta(\rho(z^*))z} = e^{(\eta_2\omega_1-\eta_1\omega_2)(\alpha^*_1\alpha_2-\alpha_1\alpha^*_2)}= e^{2{\rm i}\pi{\rm Im}(\overline zz^*)}.
\end{equation}
 The restriction to $\Lambda  \times \Lambda^*$ of
the logarithm of this ratio, that is $\log (\tilde{f}_{\rho(z^*)} (z)/\tilde{f}_{z} (\rho(z^*))),$ takes values in $ 2 \mathrm{i} \pi  \ZZ.$ 
With the notations of \cite[Lemma (10.2.3.4)]{D75}, the Hodge realization of the Weil pairing of $E$ is this restriction, 
$$\begin{matrix}
\langle\, , \, \rangle_\ZZ: & \Lambda  \otimes_\ZZ \Lambda^* &\longrightarrow & 2 \mathrm{i} \pi  \ZZ \hfill\\
&\lambda \otimes \lambda^* &\longmapsto &2\mathrm{i}\pi\mathrm{Im}(\overline{\lambda}\lambda^*)
\end{matrix}$$
and the restriction to $\Lambda  \times \Lambda^*$ of
 the ratio $\tilde{f}_{\rho(z^*)} (z)/\tilde{f}_{z} (\rho(z^*))$ is a
 trivialization of the restriction to $\Lambda \times \Lambda^* $  of the pull-back $(\exp_{E} \times \exp_{ E^*})^* \cP$ of the Poincar\'e biextension on $\Lie E_\CC \times \Lie E^*_\CC$.

\begin{remark}
	Let $N\in\mathbb N^\times$. If $P$ and $Q$ are $N$-torsion points of $E$ and $p$, $q$ their logarithms, then $\mathcal{W}_E (p,q)^N$ is an $N$-th root of unity which does not depend on the choice of $p$ and $q$, by Legendre relation. It is the $N$-th Weil pairing of the points $P$ and $Q$ ({\it see} \cite[Chap.18, \S1, pp.243-245]{L87} or/and \cite[Chap.3, \S8.]{S09}).
\end{remark}

Consider a 1-motive $M=[u:\ZZ \rightarrow G]$ defined over a sub-field $K$ of $\CC$, where $G$ is an extension of the elliptic curve $E$ by $\GG_m$ corresponding to the point $Q \in E^*(K)$, and $u(1)=R  = \exp_G(p,t) = (P, e^t f_q(p))$ is a point of $ G(K)$ living above $P \in E(K).$

\medskip

With notations of Construction \ref{B-Z} (1), $B$ is the smallest abelian sub-variety of $E \times E^* $ containing the point $(P,Q) \in (E \times E^*)(K).$ We have that:
\begin{itemize}
	\item $\dim B=2$ if and only if the points $P$ and $Q$ are	$\End(E) \otimes_\ZZ \QQ$-linearly independent;
	\item $\dim B=1$ if and only if $P$ and $Q$ are	$\End(E) \otimes_\ZZ \QQ$-linearly dependent but not both torsion points;
	\item $\dim B=0$ if and only if $P$ and $Q$ are torsion points.
\end{itemize}
Recall that by \cite{Br94} and \cite{Br95} the restriction of the Poincar\'e biextension $\cP$ to $B$ is trivial if and only if one of the two following conditions happens: $\dim B=0$  (that is $P$ and $Q$ are torsion points) or
 there exists an antisymmetric group morphism $g:E
\rightarrow E^* $ and an integer $N$ such that
$g(P)=N Q$ (and $\dim B=1$ if $P$ or $Q$ is not torsion). 

\medskip

By Construction \ref{B-Z} (2), $Z'(1)$ is the sub-torus of $\GG_m$ containing the image of the Lie bracket $[\cdot,\cdot]: B \otimes B \to \GG_m$ in \emph{loc.cit}. According to \cite[(2.8.4)]{B03} and \cite[Lemma 3.4(ii)]{B01}, for $(s_1,t_1),(s_2,t_2) \in \Lie B_\CC$ we have
$$	[(s_1,t_1),(s_2,t_2)] = \mathcal{W}_E(s_1,t_2)  \cdot \mathcal{W}_{E^*} (t_1,s_2) = \mathcal{W}_E(s_1,t_2) \cdot \mathcal{W}_E(s_2,t_1)^{-1}$$
and so
$$
\nonumber Z'(1 ) = \left\langle	
e^{2{\rm i}\pi{\rm Im}(\overline{s_1}t_2-\overline{s_2}t_1)}
\; \big\vert \;  (s_1,t_1),(s_2,t_2) \in \Lie B_\CC  \right\rangle \subseteq \GG_m. 
$$
Looking at the factor of automorphy \eqref{eq:section} of $d^*\mathcal{B}$ we observe that the torus $Z'(1)$ contains the values of the factor of automorphy of the torsor $i^*d^*\mathcal{B}.$ Thus, quotienting the fibre of $i^*d^*\mathcal{B}$ by $Z'(1),$ we get the trivial torsor that we have denoted $pr_*i^*d^* \mathcal{B}$ in Construction \ref{B-Z} (2). We have $\dim Z'(1)=0$ if and only if the torsor $i^*d^*\mathcal{B}$ is trivial (that is $i^*d^*\mathcal{B}=B \times \GG_m$), and as observed before, this happens if and only if $P$ and $Q$ are torsion points ($\dim B=0$), or $g(P)=N Q$
with $g: E \to E^*$ an antisymmetric group morphism ($\dim B=1$).

\medskip

 By Construction \ref{B-Z} (3),
the torus $Z(1)$ is the smallest sub-torus of $\GG_m$ containing $Z'(1)$ and such that the sub-torus $(Z/Z')(1)$ of $\GG_m/Z'(1)$ contains $\pi (\widetilde{R})$.
 With our hypothesis the dimension of the torus $Z(1)$ is 0 or 1. Since $\dim Z(1) = \dim Z'(1) + \dim (Z/Z')(1)$, we have three cases:
	\begin{itemize}
		\item $\dim Z(1)=0 $ implies $\dim Z'(1)=0=\dim (Z/Z') (1)$. The condition $\dim Z'(1) =0$ implies that $i^*d^*\mathcal{B}=B \times \GG_m.$
		 Moreover the subsequent condition $\dim(Z/Z')(1) =0$  implies that $\tilde{R}=\pi (\tilde{R})$ is a root of unity. 
		\item  $\dim Z(1)= \dim (Z/Z')(1)=1$ and  $ \dim Z'(1) = 0$. As before, $\dim Z'(1) =0$ implies that the torsor $i^*d^*\mathcal{B}$ is trivial.
		 Moreover the subsequent condition $\dim (Z/Z')(1) =1$  implies that $\tilde{R}=\pi (\tilde{R})$ is not a root of unity. 
		\item  $\dim Z(1)= \dim Z' (1)=1$ and  $ \dim (Z/Z')(1) = 0$. We have $\dim Z'(1) =1$ if and only if the torsor $i^*d^*\mathcal{B}$ is not trivial. The subsequent condition $\dim (Z/Z')(1) =0$ implies that $\pi (\tilde{R})$ is a root of unity. 
	\end{itemize}

\subsection{Main theorem for semi-elliptic  surfaces}\label{semielliptic}
Let $E$ be an elliptic curve defined over $\oQQ$ and consider 
 an extension $G$ of $E$ by $\GG_m$ which corresponds to a point $Q$ of $ E^*(\oQQ)$.
In this setting the semi-abelian exponential map (\ref{eq:semiabexpi}) becomes 
\begin{align}
\nonumber	\exp_{G}:  \Lie G_\CC & \longrightarrow   G (\CC) \subset (\PP^{2} \times \PP^1 )(\CC) \\
\nonumber (z, t) & \longmapsto \Big([   \wp(z):\wp'(z):1 ], [ e^{t} f_{q}(z) :1 ]\Big). 
\end{align}
 The semi-elliptic version of Conjecture \ref{WRSA-V2} is 

\begin{conjecture}[Weak Relative Semi-elliptic conjecture]\label{WRSE-V2}
 Let $E$ be an elliptic curve defined over $\oQQ$. Let $G$ be an extension of $E$ by the torus $\GG_m$, which corresponds to a point $Q $ of $ E^*(\oQQ)$, and denote by $\Omega_G$ its period matrix. Let $(z,t) $ be a point of $\Lie G_\CC\simeq\CC^2$. Then the transcendence degree over $\QQ (\Omega_G)$ of the field generated  by $z$, $t$, $\wp(z)$ and $e^{t} f_{q}(z)$  is at least the
dimension of the smallest algebraic subgroup containing a translate of the point $\exp_G(z,t)$ by an algebraic point.
\end{conjecture}

Consider now the following fields $\cE  = \bigcup_{n \geq 0} {\cE}_n$ and $\cL'=  \bigcup_{n \geq 0} {\cL}_n'$, with
\begin{equation}
\begin{aligned}
 \nonumber {\cE}_0 &= \oQQ, \quad &{\cE}_n &= \overline{{\QQ}\Big( \wp(z) ,  e^{t} f_{q}(z) \; \Big| \quad z,t \in \cE_{n-1} \Big)} \quad &\mathrm{for} \; n \geq 1,\\
 {\cL}_0' &= \overline{{\QQ}}, \quad &{\cL}_n' &= \overline{{\QQ}\Big( \Omega_G,z, t \;  \Big| \quad \wp(z),e^{t} f_{q}(z)  \in {\cL}_{n-1}'\Big)} \quad &\mathrm{for} \; n \geq 1,
\end{aligned}
\end{equation}
where the entries of $\Omega_G$ are the periods $\omega_1, \omega_2, \eta_1 ,\eta_2$ of $E$, plus the periods $ \eta_1 q - \omega_1 \zeta(q) ,\eta_2 q - \omega_2 \zeta(q)$ (see \cite[Prop 2.3]{B19}).
A variant of our main Theorem \ref{MainThm} gives

\begin{corollary}
	If Weak Relative Semi-elliptic conjecture \ref{WRSE-V2} is true, then $ {\cE}$ and $ {\cL}'$ are algebraically independent over $\oQQ.$ 
\end{corollary}

\subsection{Errata}\label{Rk-tableau}
In \cite[\S 4]{B19} the first author furnishes a table of dimensions of the motivic Galois group of the 1-motive $M=[u:\ZZ \rightarrow G]$, where $G$ is an extension of an elliptic curve $E$ by $\GG_m$. Unfortunately the second to last line is wrong and a case is missing. We give here the corrected and complete Table \ref{tableerr}.

\medskip

In Lemma 3  of \cite{PSS} one should read $\mathcal C\subset\mathcal L_n^g$ instead of $\mathcal C\subset\mathcal L_n^{2g}$. Accordingly, the definition of $\mathbf u_2$ has to be replaced by the concatenation of the elements of $\mathcal C$. Also, to have the proof of Theorem 1 in this reference work smoothly, one can define $\mathcal L_0$ as $\oQQ$ rather than $\oQQ(\widetilde{\Lambda}_A)$. This in turn gives that our earlier choice $\mathcal L_0=\oQQ(\widetilde{\Lambda}_A)$ works eventually.

\begin{table}[h]
	\caption{Dimension of $\Galmot (M)$}\label{tableerr}	
	\begin{tabular}{|c|c|c|c|c|}
		\hline
		& \rotatebox{90}{\hbox{$\dim \UR (M)$}}  & \rotatebox{90}{\hbox{$\dim \Galmot (M)$}}  & \rotatebox{90}{\hbox{$\dim \Galmot (M)$}} & $M$   \\
		&  &  $E$ CM  &  $E$ not CM & \\ \hline
		$Q$, $R$ torsion & 0 & 2& 4 & $M=[u:\ZZ \rightarrow E \times \GG_m]$ \\ 
		($\Rightarrow$ P torsion)  & & & & $ u(1)=(0,1)$ \\ \hline
		$P$, $Q$ torsion  &1 & 3& 5 &  $M=[u:\ZZ \rightarrow E \times \GG_m] $\\ 
		($R$ not torsion)  & & & &  $u(1)=(0,R)$ \\ \hline
		$R$ torsion ($\Rightarrow$ $P$ torsion)  &2 & 4& 6 & $M=[u:\ZZ \rightarrow G]$\\
		($Q$ not torsion)  & & & & $ u(1)=0$\\ \hline
		$Q$ torsion  &3 &5 & 7 & $M=[u:\ZZ \rightarrow E \times \GG_m]$\\ 
		($P$ and $R$ not torsion)  & & & & $u(1)=(P,R) $  \\ \hline
		$P$ torsion  &3 &5 &7  &  $M=[u:\ZZ \rightarrow E^* \times \GG_m] $ \\ 
		($R$ and $Q$ not torsion)  & & & & $u(1)=(Q,R) $ \\ \hline
		$P$, $Q$   & & &   & \raise-5pt\hbox{$M=[u:\ZZ \rightarrow G]$}\\\noalign{\vskip-5pt} 	
		$\End(E) \otimes_\ZZ \QQ$-lin. dep.  &2 &4 &(*) & \raise-5pt\hbox{$u(1)=R$}\\\noalign{\vskip-5pt}
		($M$ deficient) & & & &  \\ \hline
		$P$, $Q$   & & & & \raise-5pt\hbox{$M=[u:\ZZ \rightarrow G]$}\\\noalign{\vskip-5pt}
		$\End(E) \otimes_\ZZ \QQ$-lin. dep.  &3 &5 &7 & \raise-5pt\hbox{$u(1)=R$}\\\noalign{\vskip-5pt} 
		($M$ not deficient) & & & &  \\ \hline
		$P$, $Q$   &5 &7 & 9  & $M=[u:\ZZ \rightarrow G]$\\ 
		$\End(E) \otimes_\ZZ \QQ$-lin. indep.  & & & &$ u(1)=R$\\ \hline
	\end{tabular}
\par\bigskip (*) Observe that if the elliptic curve is not CM, the 1-motive $M$ cannot be deficient.
\end{table}

\medskip


\bibliographystyle{plain}

\end{document}